\documentclass[12pt,]{article}
\usepackage[margin=1in]{geometry}
\usepackage{amscd}
\usepackage{amssymb,amsmath,amsfonts,amsbsy,latexsym}
\usepackage{amsmath}
\usepackage{amsthm}
\usepackage{enumerate}
\usepackage{tikz}
\usetikzlibrary{external}
\usepackage{mathrsfs}
\usepackage{subcaption}
\usepackage{pgfplots}
\usepackage{wrapfig, framed, caption}
\usepackage{float}
\usetikzlibrary{arrows}
\usepackage[font=small,labelfont=bf]{caption}
\usepackage{xcolor}
\usepackage{mathtools}
\usepackage[colorlinks=true, linkcolor=blue, citecolor=blue,
pagebackref=true]{hyperref}
\usepackage{fouriernc}
\usepackage[cal=cm,scr=rsfs]{mathalfa}

\usepackage{ulem}
\usepackage{comment}
\usepackage[capitalise]{cleveref}

\makeatletter
\newcommand{\refcheckize}[1]{%
  \expandafter\let\csname @@\string#1\endcsname#1%
  \expandafter\DeclareRobustCommand\csname relax\string#1\endcsname[1]{%
    \csname @@\string#1\endcsname{##1}\@for\@temp\coloneqq##1\do{\wrtusdrf{\@temp}\wrtusdrf{{\@temp}}}}%
  \expandafter\let\expandafter#1\csname relax\string#1\endcsname
}

\newtheorem{theorem}{Theorem}[section]
\newtheorem*{theorem*}{Theorem}

\newtheorem*{theoremY*}{Theorem Y}

\newtheorem*{theoremAB*}{Theorem AB}

\newtheorem{corollary}{Corollary}[section]
\newtheorem*{corollary*}{Corollary}
\newtheorem{proposition}{Proposition}[section]
\newtheorem{lemma}{Lemma}[section]

\newtheorem*{claim*}{Claim}

\theoremstyle{definition}
\newtheorem{definition}{Definition}[section]
\theoremstyle{remark}
\newtheorem{remark}{Remark}[section]
\newtheorem*{remark*}{Remark}

\theoremstyle{remarks}

\theoremstyle{step}

\allowdisplaybreaks

\title{Metrical theory of signed Engel expansions}

\author{Can Wang\\ School of Mathematics and Statistics, Wuhan University \\  Wuhan, Hubei 430072, 
	PR China \\ email: wangcan10010@whu.edu.cn}
\begin{document}

\frenchspacing
\maketitle

\begin{abstract}
	Motivated by the Engel and Pierce expansions, we introduce a signed Engel expansion. We expand each $x\in(0,1)\setminus\mathbb{Q}$ uniquely as 
	$$x=\frac{\epsilon_{1}(x)}{d_{1}(x)}+\frac{\epsilon_{2}(x)}{d_{1}(x)d_{2}(x)}+\cdots+\frac{\epsilon_{n}(x)}{d_{1}(x)d_{2}(x)\cdots d_{n}(x)}+\cdots,$$
	where $\epsilon_{1}(x)\coloneqq1$ and $\epsilon_{n}(x)\in\left\{1,-1\right\}$ for $n\geq2$. The digit sequence $\left\{d_{n}(x)\right\}_{n\geq1}$ satisfying  $d_{n+1}(x)\geq d_{n}(x)+2$ when $\epsilon_{n+1}(x)=-\epsilon_{n}(x)$ forms a non-decreasing sequence of even positive integers tending to infinity. On the one hand, we obtain the law of large numbers, the central limit theorem and the law of the iterated logarithm regarding $d_{n}(x)$ and $\Delta_{n}(x)\coloneqq d_{n}(x)-d_{n-1}(x)\ (n\geq2)\ (\Delta_{1}(x)\coloneqq d_{1}(x))$. On the other hand, we prove a Borel--Bernstein theorem on the zero-one law on the Lebesgue measure of the set 
	$$\left\{x\in(0,1)\colon R_{n}(x)\geq\phi(n)\  \textnormal{ for infinity many } n\right\},$$ 
	where $R_{n}(x)\coloneqq\frac{d_{n}(x)}{d_{n-1}(x)}\ (n\geq2)\ (R_{1}(x)\coloneqq d_{1}(x))$ and $\phi$ is an arbitrary positive function defined on the set of positive integers.
\end{abstract}
  
  
%
%

\section{Introduction} \label{sec:intro}
There exists a variety of ways to represent a real number in the form of infinite series, such as continued fraction expansion \cite{MR161833}, Lüroth expansion \cite{MR1510205}, Cantor expansion \cite{cantor1869ueber}, Oppenheim expansion \cite{MR309877}, Engel expansion \cite{MR102496}, and Pierce expansion \cite{MR1521866,MR825872}. For more infinite series representations of real numbers, see Galambos' monograph \cite{MR568141}. Based on the Engel and Pierce expansions, we introduce a new infinite series representation for real numbers, which we term the signed Engel expansion. Prior to introducing this new representation, we provide a brief review of the Engel and Pierce expansions.\\
\indent The Engel expansion can be generated by the interval map $T^{E}\colon[0,1)\to[0,1)$ defined by 
\begin{equation}\nonumber
	T^{E}(0)\coloneq0,\quad\textnormal{and}\quad T^{E}(x)\coloneq\left\lceil\frac{1}{x}\right\rceil x-1\ \textnormal{ for }\ x\in(0,1),
\end{equation}
where $\lceil x\rceil$ denotes the smallest integer not less than $x$. For the image of $T^{E}$, see Fig.~\ref{fig:TEandTP}(\subref{fig:TE}).
\begin{figure}[h]
	\centering
	\begin{subfigure}[b]{0.45\textwidth}
		\centering
	\begin{tikzpicture}[scale=0.015]
		\draw[black] (0,0) -- (0,420);
		\draw[black] (420,0) -- (420,420);
		\draw[black] (0,420) -- (420,420);
		\draw[black] (0,0) -- (420,0);
		\draw[dashed] (0,0) -- (420,420);
		\draw[black] (210,0) -- (420,420);
		\draw[dashed] (210,0) -- (210,210);
		\draw[black] (140,0) -- (210,210);
		\draw[dashed] (140,0) -- (140,140);
		\draw[black] (105,0) -- (140,140);
		\draw[dashed] (105,0) -- (105,105);
		\draw[black] (84,0) -- (105,105);
		\draw[dashed] (84,0) -- (84,84);
		\draw[black] (70,0) -- (84,84);
		\draw[dashed] (70,0) -- (70,70);
		\draw[black] (60,0) -- (70,70);
		\draw[dashed] (60,0) -- (60,60);
		\draw[black] (52.5,0) -- (60,60);
		\draw (420,420) circle (4);
		\draw (210,210) circle (4);
		\draw (140,140) circle (4);
		\draw (105,105) circle (4);
		\draw (84,84) circle (4);
		\draw (70,70) circle (4);
		\draw (60,60) circle (4);
		\filldraw (210,0) circle (3);
		\filldraw (140,0) circle (3);
		\filldraw (105,0) circle (3);
		\filldraw (84,0) circle (3);
		\filldraw (70,0) circle (3);
		\filldraw (60,0) circle (3);
		\filldraw (0,0) circle (3);
		\coordinate[label=below:$1$](1) at (420,0);
		\coordinate[label=below:$\frac{1}{2}$](1/2) at (210,0);
		\coordinate[label=below:$\frac{1}{3}$](1/3) at (140,0);
		\coordinate[label=below:$\frac{1}{4}$](1/4) at (105,0);
		\coordinate[label=below:$\frac{1}{5}$](1/5) at (84,0);
		\coordinate[label=below:$\frac{1}{6}$](1/6) at (70,0);
		\coordinate[label=below:$\frac{1}{7}$](1/7) at (60,0);
		\coordinate[label=below:$0$](0) at (0,0);
		\coordinate[label=left:$1$](1) at (0,420);
	\end{tikzpicture}
	\caption{The map $T^{E}$}
	\label{fig:TE}
\end{subfigure}
\begin{subfigure}[b]{0.45\textwidth}
	\centering
		\begin{tikzpicture}[scale=0.015]
		\draw[black] (0,0) -- (0,420);
		\draw[black] (420,0) -- (420,420);
		\draw[black] (0,420) -- (420,420);
		\draw[black] (0,0) -- (420,0);
		\draw[dashed] (0,0) -- (420,420);
		\draw[black] (210,210) -- (420,0);
		\draw[dashed] (210,0) -- (210,210);
		\draw[black] (140,140) -- (210,0);
		\draw[dashed] (140,0) -- (140,140);
		\draw[black] (105,105) -- (140,0);
		\draw[dashed] (105,0) -- (105,105);
		\draw[black] (84,84) -- (105,0);
		\draw[dashed] (84,0) -- (84,84);
		\draw[black] (70,70) -- (84,0);
		\draw[dashed] (70,0) -- (70,70);
		\draw[black] (60,60) -- (70,0);
		\draw[dashed] (60,0) -- (60,60);
		\draw[black] (52.5,52.5) -- (60,0);
		\draw (420,0) circle (4);
		\draw (210,210) circle (4);
		\draw (140,140) circle (4);
		\draw (105,105) circle (4);
		\draw (84,84) circle (4);
		\draw (70,70) circle (4);
		\draw (60,60) circle (4);
		\filldraw (210,0) circle (3);
		\filldraw (140,0) circle (3);
		\filldraw (105,0) circle (3);
		\filldraw (84,0) circle (3);
		\filldraw (70,0) circle (3);
		\filldraw (60,0) circle (3);
		\filldraw (0,0) circle (3);
		\coordinate[label=below:$1$](1) at (420,0);
		\coordinate[label=below:$\frac{1}{2}$](1/2) at (210,0);
		\coordinate[label=below:$\frac{1}{3}$](1/3) at (140,0);
		\coordinate[label=below:$\frac{1}{4}$](1/4) at (105,0);
		\coordinate[label=below:$\frac{1}{5}$](1/5) at (84,0);
		\coordinate[label=below:$\frac{1}{6}$](1/6) at (70,0);
		\coordinate[label=below:$\frac{1}{7}$](1/7) at (60,0);
		\coordinate[label=below:$0$](0) at (0,0);
		\coordinate[label=left:$1$](1) at (0,420);
	\end{tikzpicture}
	\caption{The map $T^{P}$}
	\label{fig:TP}
	\end{subfigure}
	\caption{}
	\label{fig:TEandTP}
\end{figure}
For any $x\in(0,1)$, the digit sequence $\left\{d^{E}_{n}(x)\right\}_{n\geq1}$ of its Engel expansion is defined as
\begin{equation}\nonumber
	d^{E}_{1}(x)\coloneq
	\left\lceil\frac{1}{x}\right\rceil,\quad\textnormal{and}\quad d^{E}_{n+1}(x)\coloneq d^{E}_{1}\left(\left(T^{E}\right)^{n}(x)\right)\textnormal{ for all }n\in\mathbb{N},
\end{equation}
where $\left(T^{E}\right)^{n}$ denotes the $n$-th iteration
of $T^{E}$. 
Then, every rational number $x\in(0,1)$ can be written as a finite Engel expansion of the form
$$x=\frac{1}{d^{E}_{1}(x)}+\frac{1}{d^{E}_{1}(x)d^{E}_{2}(x)}+\cdots+\frac{1}{d^{E}_{1}(x)d^{E}_{2}(x)\cdots d^{E}_{n}(x)},$$
and every irrational number $x\in(0,1)$ can be written as an infinite Engel expansion in the form of
$$x=\frac{1}{d^{E}_{1}(x)}+\frac{1}{d^{E}_{1}(x)d^{E}_{2}(x)}+\cdots+\frac{1}{d^{E}_{1}(x)d^{E}_{2}(x)\cdots d^{E}_{n}(x)}+\cdots.$$
It can be checked that the digit sequence satisfies $2\leq d^{E}_{1}(x)\leq d^{E}_{2}(x)\leq\cdots\leq d^{E}_{n}(x)$ for all $n\in\mathbb{N}$. Moreover, $d^{E}_{n}(x)\to\infty$ as $n\to\infty$ if $x$ is irrational. For further details on Engel expansion, we refer to \cite[p. 7]{MR102496} and \cite[p. 17]{MR568141}.\\
\indent  Define the interval map $T^{P}\colon[0,1)\to[0,1)$ related to Pierce expansion as
\begin{equation}\nonumber
	T^{P}(0)\coloneq0,\quad\textnormal{and}\quad T^{E}(x)\coloneq1-\left\lfloor\frac{1}{x}\right\rfloor x\ \textnormal{ for }\ x\in(0,1),
\end{equation}
where $\lfloor x\rfloor$ denotes the largest integer not exceeding $x$. See Fig.~\ref{fig:TEandTP}(\subref{fig:TP}) for the image of $T^{P}$.
Similarly, for any $x\in(0,1)$, the digit sequence $\left\{d^{P}_{n}(x)\right\}_{n\geq1}$ of its Pierce expansion is defined as
\begin{equation}\nonumber
	d^{P}_{1}(x)\coloneq
	\left\lfloor\frac{1}{x}\right\rfloor,\quad\textnormal{and}\quad d^{P}_{n+1}(x)\coloneq d^{P}_{1}\left(\left(T^{P}\right)^{n}(x)\right)\textnormal{ for all }n\in\mathbb{N},
\end{equation}
where $\left(T^{P}\right)^{n}$ denotes $n$-th iteration of $T^{P}$. Then, every rational number $x\in(0,1)$ has a finite Pierce expansion of the form
$$x=\frac{1}{d^{P}_{1}(x)}+\frac{-1}{d^{P}_{1}(x)d^{P}_{2}(x)}+\cdots+\frac{\left(-1\right)^{n-1}}{d^{p}_{1}(x)d^{P}_{2}(x)\cdots d^{p}_{n}(x)},$$
and every irrational number $x\in(0,1)$ admits an infinite Pierce expansion of the form
$$x=\frac{1}{d^{P}_{1}(x)}+\frac{-1}{d^{P}_{1}(x)d^{P}_{2}(x)}+\cdots+\frac{\left(-1\right)^{n-1}}{d^{P}_{1}(x)d^{P}_{2}(x)\cdots d^{P}_{n}(x)}+\cdots.$$
It can be checked that $1\leq d^{P}_{1}(x)<d^{P}_{2}(x)<\cdots<d^{P}_{n}(x)$ for all $n\in\mathbb{N}$. Moreover, if $x$ is rational and Pierce expansion of $x$ has exactly $n\ (\geq2)$ terms then $d^{P}_{n}(x)\geq d^{P}_{n-1}(x)+2$, and if $x$ is irrational then $d^{P}_{n}(x)\to\infty$ as $n\to\infty$. For more details, we refer to \cite[p. 23-24]{MR825872}.\\
\indent Now, let us present the signed Engel expansion. Define an interval map $T\colon[0,1)\to[0,1)$ as
\begin{equation}\nonumber
	Tx\coloneq
	\begin{cases}
		\left\lceil\frac{1}{x}\right\rceil x-1, & \text{ if }\ x\in\left(\frac{1}{2k},\frac{1}{2k-1}\right),\\
		(-1)(\left\lfloor\frac{1}{x}\right\rfloor x-1), & \text{ if }\ x\in\left(\frac{1}{2k+1},\frac{1}{2k}\right),\\
		0, & \text{ if }\ x\in\left\{0\right\}\cup\left\{\frac{1}{n}:n\in\mathbb{N}\setminus\left\{1\right\}\right\},
	\end{cases}
\end{equation}
where $k\in\mathbb{N}$. The map $T$ is illustrated in \cref{fig:T}.
\begin{figure}[h]
	\centering
	\begin{tikzpicture}[scale=0.016]
		\draw[black] (0,0) -- (0,420);
		\draw[black] (420,0) -- (420,420);
		\draw[black] (0,420) -- (420,420);
		\draw[black] (0,0) -- (420,0);
		\draw[dashed] (0,0) -- (420,420);
		\draw[black] (210,0) -- (420,420);
		\draw[dashed] (210,0) -- (210,210);
		\draw[black] (140,140) -- (210,0);
		\draw[dashed] (140,0) -- (140,140);
		\draw[black] (105,0) -- (140,140);
		\draw[dashed] (105,0) -- (105,105);
		\draw[black] (84,84) -- (105,0);
		\draw[dashed] (84,0) -- (84,84);
		\draw[black] (70,0) -- (84,84);
		\draw[dashed] (70,0) -- (70,70);
		\draw[black] (60,60) -- (70,0);
		\draw[dashed] (60,0) -- (60,60);
		\draw[black] (52.5,0) -- (60,60);
		\draw (420,420) circle (4);
		\draw (140,140) circle (4);
		\draw (84,84) circle (4);
		\draw (60,60) circle (4);
		\filldraw (140,0) circle (4);
		\filldraw (84,0) circle (4);
		\filldraw (60,0) circle (4);
		\filldraw (0,0) circle (4);
		\coordinate[label=below:$1$](1) at (420,0);
		\coordinate[label=below:$\frac{1}{2}$](1/2) at (210,0);
		\coordinate[label=below:$\frac{1}{3}$](1/3) at (140,0);
		\coordinate[label=below:$\frac{1}{4}$](1/4) at (105,0);
		\coordinate[label=below:$\frac{1}{5}$](1/5) at (84,0);
		\coordinate[label=below:$\frac{1}{6}$](1/6) at (70,0);
		\coordinate[label=below:$\frac{1}{7}$](1/7) at (60,0);
		\coordinate[label=below:$0$](0) at (0,0);
		\coordinate[label=left:$1$](1) at (0,420);
	\end{tikzpicture}
	\caption{The map $T$}
	\label{fig:T}
\end{figure}
For any $x\in(0,1)$ and $k\in\mathbb{N}$, define
\begin{equation}\nonumber
	d_{1}(x)\coloneq
	\begin{cases}
		\left\lceil\frac{1}{x}\right\rceil, & \text{ if }\ x\in\left[\frac{1}{2k},\frac{1}{2k-1}\right),\\
		\left\lfloor\frac{1}{x}\right\rfloor, & \text{ if }\ x\in\left[\frac{1}{2k+1},\frac{1}{2k}\right),
	\end{cases}
	\quad\textnormal{and}\quad
	s_{1}(x)\coloneq
	\begin{cases}
		1, & \text{ if }\ x\in\left[\frac{1}{2k},\frac{1}{2k-1}\right),\\
		-1, & \text{ if }\ x\in\left[\frac{1}{2k+1},\frac{1}{2k}\right).
	\end{cases}
\end{equation}
Let $s_{n+1}(x)\coloneq s_{1}(T^{n}x)$ for all $n\in\mathbb{N}$, where $T^{n}$ denotes the $n$-th iteration of $T$. Then, we define respectively the digit sequence $\left\{d_{n}(x)\right\}_{n\geq1}$ and the sign sequence $\left\{\epsilon_{n}(x)\right\}_{n\geq1}$ of the signed Engel expansion of $x$ as
\begin{equation}\nonumber
	d_{n+1}(x)\coloneq d_{1}(T^{n}x),
\end{equation}
and
\begin{equation}\nonumber
	\epsilon_{1}(x)\coloneq1,\ \ \epsilon_{n+1}(x)\coloneq\prod_{k=1}^{n}s_{k}(x).
\end{equation}
With these notations, for each $x\in(0,1)$, we have $Tx=s_{1}(x)\left(d_{1}(x)x-1\right)$. So,
\begin{equation}\label{eq:recursive formula of signed Engel expansion}
	x=\frac{\epsilon_{1}(x)}{d_{1}(x)}+\frac{s_{1}(x)}{d_{1}(x)}Tx.
\end{equation}
If $Tx=0$, then $x=\frac{1}{d_{1}(x)}$. Otherwise, $Tx\neq 0$, and we can replace $x$ with $Tx$ in \cref{eq:recursive formula of signed Engel expansion} to obtain 
$$x=\frac{\epsilon_{1}(x)}{d_{1}(x)}+\frac{s_{1}(x)}{d_{1}(x)}\left(\frac{\epsilon_{1}(Tx)}{d_{1}(Tx)}+\frac{s_{1}(Tx)}{d_{1}(Tx)}T^{2}x\right)=\frac{\epsilon_{1}(x)}{d_{1}(x)}+\frac{\epsilon_{2}(x)}{d_{1}(x)d_{2}(x)}+\frac{\epsilon_{3}(x)}{d_{1}(x)d_{2}(x)}T^{2}x.$$
Assume that $T^{1}x\neq0,\cdots,T^{n-1}x\neq0$ for some $n\geq2$. Then, by performing similar iterative processes, we have
\begin{equation}\label{eq:signed Engel expansion}
	x=\frac{\epsilon_{1}(x)}{d_{1}(x)}+\frac{\epsilon_{2}(x)}{d_{1}(x)d_{2}(x)}+\cdots+\frac{\epsilon_{n}(x)}{d_{1}(x)d_{2}(x)\cdots d_{n}(x)}+\frac{\epsilon_{n+1}(x)}{d_{1}(x)d_{2}(x)\cdots d_{n}(x)}T^{n}x.
\end{equation}
If $T^{n}x=0$, then the iteration terminates and the last term on the right-hand side of \cref{eq:signed Engel expansion} vanishes. Otherwise, that is, $T^{1}x\neq0,\cdots,T^{n-1}x\neq0$, and $T^{n}x\neq0$, then we can perform a further iteration on \cref{eq:signed Engel expansion} by replacing $x$ with $T^{n}x$ in \cref{eq:recursive formula of signed Engel expansion}.
Suppose that $T^{n}x\neq 0$ for all $n\in\mathbb{N}$. Then the above iterative process can continue constantly. By the algorithm above, we see that $x\in\mathbb{Q}\cap(0,1)$ if and only if $x$ can be represented uniquely as
\begin{equation}\label{eq:finite signed Engel expansion}
	x=\frac{\epsilon_{1}(x)}{d_{1}(x)}+\frac{\epsilon_{2}(x)}{d_{1}(x)d_{2}(x)}+\cdots+\frac{\epsilon_{n}(x)}{d_{1}(x)d_{2}(x)\cdots d_{n}(x)},
\end{equation}
and $x\in(0,1)\setminus\mathbb{Q}$ if and only if $x$ can be written uniquely in the form of
\begin{equation}\label{eq:infinite signed Engel expansion}
	x=\frac{\epsilon_{1}(x)}{d_{1}(x)}+\frac{\epsilon_{2}(x)}{d_{1}(x)d_{2}(x)}+\cdots+\frac{\epsilon_{n}(x)}{d_{1}(x)d_{2}(x)\cdots d_{n}(x)}+\cdots,
\end{equation}
where $\epsilon_{1}(x)=1$ and $\epsilon_{n}\in\left\{1,-1\right\}$ for any $n\geq2$. Further, the sequence $\left\{d_{n}(x)\right\}_{n\geq1}$ satisfying  $d_{n+1}(x)\geq d_{n}(x)+2$ if $\epsilon_{n+1}(x)=-\epsilon_{n}(x)$, forms a non-decreasing sequence of positive integers and if $x$ is irrational, then $d_{n}(x)\to\infty$ as $n\to\infty$. See \cref{pro:signed Engel Expansion} in \cref{preliminaries}.\\
\indent Let us return to the Engel expansion. Borel \cite{MR23007} asserted that $\lim_{n\to\infty}\frac{\log d_{n}^{E}(x)}{n}=1$ holds for Lebesgue almost all $x\in(0,1)$. L\'{e}vy \cite{MR23008} announced a central limit theorem and a law of the iterated logarithm for the digit sequence $\left\{ d_{n}^{E}\right\}_{n\geq1}$, and also sketched the reasoning behind the three limit theorems. By noting that the digit sequence $\left\{d_{n}^{E}\right\}_{n\geq1}$
forms a time-homogeneous Markov chain, Erd\H os, R\'{e}nyi and Sz\"{u}sz \cite{MR102496} provided detailed proofs for these limit theorems. Williams \cite{MR334317} provided a more concise and fundamental proof for these limit theorems by replacing $\left\{d_{n}^{E}\right\}_{n\geq1}$ with a new sequence that shares the same initial distribution and one step transition probability. Galambos \cite{MR1180497} extended the limit theorems to general Oppenheim expansions.\\
\indent Regarding signed Engel expansion, we obtain the following analogous limit theorems. We denote by $\mathcal{L}$ the Lebesgue measure on the interval $(0,1)$, and by $\mathbb{I}$ the set of irrational numbers in $(0,1)$, that is, $\mathbb{I}\coloneq(0,1)\setminus\mathbb{Q}$.
\begin{theorem}\label{limit theorems}
	For signed Engel expansions, the following limit theorems hold.
	\begin{enumerate}[(1)]
		\item\label{LLN for dn}
		 Law of large numbers (LLN): For Lebesgue almost all $x\in(0,1)$,
		\begin{equation}\nonumber
			\lim_{n\to\infty}\frac{\log d_{n}(x)}{n}=1.
		\end{equation}
		\item Central limit theorem (CLT): For any $t\in\mathbb{R}$,
		\begin{equation}\nonumber
			\lim_{n\to\infty}\mathcal{L}\left\{x\in\mathbb{I}\colon \frac{\log d_{n}(x)-n}{\sqrt{n}}\leq t\right\}=\frac{1}{\sqrt{2\pi}}\int_{-\infty}^{t}e^{-\frac{u^{2}}{2}}du.
		\end{equation}
		\item Law of the iterated logarithm (LIL): For Lebesgue almost all $x\in(0,1)$,
		\begin{equation}\nonumber
			\varlimsup_{n\to\infty}\frac{\log d_{n}(x)-n}{\sqrt{2n\log\log n}}=1,\quad\textnormal{ and }\quad\varliminf_{n\to\infty}\frac{\log d_{n}(x)-n}{\sqrt{2n\log\log n}}=-1.
		\end{equation}
	\end{enumerate}
\end{theorem}
For any $x\in\mathbb{I}$ and $n\in\mathbb{N}$, define the gap sequence $\left\{\Delta_{n}(x)\right\}_{n\geq1}$ as 
\begin{equation}\nonumber
	\Delta_{1}(x)\coloneq d_{1}(x),\quad\textnormal{and}\quad\Delta_{n}(x)\coloneq d_{n}(x)-d_{n-1}(x)\ \textnormal{for any}\ n\geq2.
\end{equation}
For the gap sequence $\left\{\Delta_{n}(x)\right\}_{n\geq1}$, we also establish a result analogous to \cref{limit theorems}, as follows. For the case of Pierce expansion, the corresponding results can be found in \cite{MR825872} and \cite{MR4954368}.
\begin{corollary}\label{limit theorems for gap sequences}
	For signed Engel expansions, the following limit theorems hold.
	\begin{enumerate}[(1)]
		\item\label{log Deltan/n to 1 a.e.} Law of large numbers (LLN): For Lebesgue almost all $x\in(0,1)$,
		\begin{equation}\nonumber
			\lim_{n\to\infty}\frac{\log \Delta_{n}(x)}{n}=1.
		\end{equation}
		\item Central limit theorem (CLT): For any $t\in\mathbb{R}$,
		\begin{equation}\nonumber
			\lim_{n\to\infty}\mathcal{L}\left\{x\in\mathbb{I}\colon \frac{\log \Delta_{n}(x)-n}{\sqrt{n}}\leq t\right\}=\frac{1}{\sqrt{2\pi}}\int_{-\infty}^{t}e^{-\frac{u^{2}}{2}}du.
		\end{equation}
		\item Law of the iterated logarithm (LIL): For Lebesgue almost all $x\in(0,1)$,
		\begin{equation}\nonumber
			\varlimsup_{n\to\infty}\frac{\log \Delta_{n}(x)-n}{\sqrt{2n\log\log n}}=1,\quad\textnormal{ and }\quad\varliminf_{n\to\infty}\frac{\log \Delta_{n}(x)-n}{\sqrt{2n\log\log n}}=-1.
		\end{equation}
	\end{enumerate}
\end{corollary}
For any $x\in\mathbb{I}$ and $n\in\mathbb{N}$, define the ratio sequence $\left\{R_{n}(x)\right\}_{n\geq1}$ as 
\begin{equation}\nonumber
	R_{1}(x)\coloneq d_{1}(x),\quad\textnormal{and}\quad R_{n}(x)\coloneq\frac{d_{n}(x)}{d_{n-1}(x)}\ \textnormal{for any}\ n\geq2.
\end{equation}
 Let $\phi\colon\mathbb{N}\to\left(0,\infty\right)$ be an arbitrary function. Define
\begin{equation}\nonumber
	R(\phi)\coloneq\left\{x\in\mathbb{I}\colon R_{n}(x)\geq \phi(n),\textnormal{ i.m. }n\right\},
\end{equation}
where i.m. denotes infinitely many. In the context of Oppenheim expansion, the set corresponding to $R(\phi)$ was studied by Galambos in \cite{MR568142}. Motivated by the Oppenheim expansion case, we prove the Borel--Bernstein theorem on the zero-one law of the Lebesgue measure for the set $R(\phi)$.
\begin{theorem}\label{Borel-Bernstain thm of ratio form}
	Let $\phi\colon\mathbb{N}\to\left(0,\infty\right)$ be an arbitrary function. Then 
	\begin{equation}\nonumber
		\mathcal{L}\left(R(\phi)\right)=
		\begin{cases}
			1, &\textnormal{ if }\ \sum_{n=1}^{\infty}\frac{1}{\phi(n)}=\infty,\\
			0, &\textnormal{ if }\ \sum_{n=1}^{\infty}\frac{1}{\phi(n)}<\infty.
		\end{cases}
	\end{equation}
\end{theorem}
According to whether the series $\sum_{n=1}^{\infty}\frac{1}{\phi(n)}$ converges, we compute the limsup of $\frac{R_{n}(x)}{\phi(n)}$ in the sense of Lebesgue almost everywhere.
\begin{corollary}\label{lebesguemeasureoflimsupRn/phin}
	Let $\phi\colon\mathbb{N}\to\left(0,\infty\right)$ be a function. The following results hold.
	\begin{enumerate}[(1)]
		\item If $\sum_{n=1}^{\infty}\frac{1}{\phi(n)}<\infty$, then, for Lebesgue almost all $x\in(0,1)$, we have $\varlimsup_{n\to\infty}\frac{R_{n}(x)}{\phi(n)}=0$.
		\item If $\sum_{n=1}^{\infty}\frac{1}{\phi(n)}=\infty$, then, for Lebesgue almost all $x\in(0,1)$, we have  $\varlimsup_{n\to\infty}\frac{R_{n}(x)}{\phi(n)}=\infty$.
	\end{enumerate}
\end{corollary}
For any $x\in\mathbb{I}$ and $n\in\mathbb{N}$, define
\begin{equation}\nonumber
	M_{n}(x)\coloneqq\max\left\{R_{k}(x)\colon1\leq k\leq n\right\}.
\end{equation}
Similar to $R(\phi)$, we define
\begin{equation}\nonumber
	M(\phi)\coloneq\left\{x\in\mathbb{I}\colon M_{n}(x)\geq\phi(n),\textnormal{ i.m. }n\right\}.
\end{equation}
When $\phi(n)$ is non-decreasing in $n$, a result similar to \cref{Borel-Bernstain thm of ratio form} holds for $M(\phi)$.
\begin{corollary}\label{Borel--Cantelli lemma for Mphi}
	Let $\phi\colon\mathbb{N}\to\left(0,\infty\right)$ be a non-decreasing function.
	Then
	\begin{equation}\nonumber
		\mathcal{L}\left(M(\phi)\right)=
		\begin{cases}
			1, &\textnormal{ if }\ \sum_{n=1}^{\infty}\frac{1}{\phi(n)}=\infty,\\
			0, &\textnormal{ if }\ \sum_{n=1}^{\infty}\frac{1}{\phi(n)}<\infty.
		\end{cases}
	\end{equation}
\end{corollary}
Regarding the upper limit of $\frac{M_{n}(x)}{\phi(n)}$, there is also a result similar to \cref{lebesguemeasureoflimsupRn/phin}.
\begin{corollary}\label{lebesguemeasureoflimsupMn/phin}
	Let $\phi\colon\mathbb{N}\to\left(0,\infty\right)$ be a non-decreasing function. The following results hold.
	\begin{enumerate}[(1)]
		\item If $\sum_{n=1}^{\infty}\frac{1}{\phi(n)}<\infty$, then, for Lebesgue almost all $x\in(0,1)$, we have $\varlimsup_{n\to\infty}\frac{M_{n}(x)}{\phi(n)}=0$.
		\item If $\sum_{n=1}^{\infty}\frac{1}{\phi(n)}=\infty$, then, for Lebesgue almost all $x\in(0,1)$, we have  $\varlimsup_{n\to\infty}\frac{M_{n}(x)}{\phi(n)}=\infty$.
	\end{enumerate}
\end{corollary}
The following result is an application of \cref{Borel-Bernstain thm of ratio form} and \cref{Borel--Cantelli lemma for Mphi} when $\phi(n)$ is taken as a special class of functions $n(\log n)^{\alpha}$ with $\alpha>1$ or $\alpha<1$.
\begin{theorem}\label{limsuplogRn-logn/loglogn}
	For Lebesgue almost all $x\in(0,1)$, we have
	\begin{equation}\nonumber
		\varlimsup_{n\to\infty}\frac{\log R_{n}(x)-\log n}{\log\log n}=1,\quad\textnormal{and}\quad\varlimsup_{n\to\infty}\frac{\log M_{n}(x)-\log n}{\log\log n}=1.
	\end{equation}
\end{theorem}
Noting that \cref{limsuplogRn-logn/loglogn} characterizes the upper limits, the following result provides the lower limits.
\begin{theorem}\label{liminflogRn-logn/loglogn}
	For Lebesgue almost all $x\in(0,1)$, we have
	\begin{equation}\nonumber
		\varliminf_{n\to\infty}\frac{\log R_{n}(x)-\log n}{\log\log n}=-\infty,\quad\textnormal{and}\quad\varliminf_{n\to\infty}\frac{\log M_{n}(x)-\log n}{\log\log n}=0.
	\end{equation}
\end{theorem}
The following corollary characterizes the growth rate of $M_{n}$.
\begin{corollary}\label{limlogMn/logn}
	For Lebesgue almost all $x\in(0,1)$, we have
	\begin{equation}\nonumber
		\lim_{n\to\infty}\frac{\log M_{n}(x)}{\log n}=1.
	\end{equation}
\end{corollary}
The rest of this paper is organized as follows. In \cref{preliminaries}, we first present the proof of \cref{eq:finite signed Engel expansion,eq:infinite signed Engel expansion}, and then introduce the concept of the symbolic space for signed Engel expansions, along with the properties of the basic intervals. In \cref{proof of the results}, we complete the proofs of the aforementioned theorems and corollaries.
\section{Preliminaries}\label{preliminaries}
    This section begins by presenting a rigorous proof of \cref{eq:finite signed Engel expansion,eq:infinite signed Engel expansion}. Subsequently, based on the characteristics of digit sequences in the signed Engel expansion, we establish some fundamental properties associated with symbolic space and basic intervals. Finally, we conclude by deriving the lengths of these basic intervals and a conditional probability formula.
\begin{proposition}\label{pro:signed Engel Expansion}
	 Equations \eqref{eq:finite signed Engel expansion} and \eqref{eq:infinite signed Engel expansion} hold. In addition,  for any $x_{0}\in(0,1)$, if $d_{k}(x_{0})$ is odd for some $k\in\mathbb{N}$, then $x_{0}$ is rational and the finite sum on the right-hand side of \cref{eq:finite signed Engel expansion} has exactly $k$ terms, that is
	\begin{equation}\label{eq:finity odd signed Engel expansion}
		x_{0}=\frac{1}{d_{1}(x_{0})}+\frac{\epsilon_{2}(x_{0})}{d_{1}(x_{0})d_{2}(x_{0})}+\cdots+\frac{\epsilon_{k}(x_{0})}{d_{1}(x_{0})d_{2}(x_{0})\cdots d_{k}(x_{0})}.
	\end{equation}
\end{proposition}
\begin{remark}
	By \cref{eq:infinite signed Engel expansion} and \cref{pro:signed Engel Expansion}, for irrational $x$, $\{d_{n}(x)\}_{n\geq1}$ is a non-decreasing sequence of positive even integers tending to infinity, with $d_{n+1}(x)\geq d_{n}(x)+2$ whenever $\epsilon_{n+1}(x)=-\epsilon_{n}(x)$.
\end{remark}
\begin{proof}
	For each $x\in(0,1)$, noting that $d_{1}(x)\geq2$, $d_{n+1}(x)=d_{1}(T^{n}x)$, $Tx\leq x$ and that $d_{1}(x)$ is non-increasing in $x$, we have $2\leq d_{1}(x)\leq d_{2}(x)\cdots\leq d_{n}(x)$. So, $\left\{d_{n}(x)\right\}_{n\geq1}$ forms a non-decreasing sequence of positive integers. If $d_{n+1}(x)=d_{n}(x)$ or $d_{n+1}(x)=d_{n}(x)+1$, then there exists $k\in\mathbb{N}$, such that $T^{n}x,\ T^{n-1}x\in\left[\frac{1}{2k+1},\frac{1}{2k-1}\right)$, where $T^{0}$ denotes the identity mapping. It follows that $T^{n-1}x\in\left(\frac{1}{2k},\frac{1}{2k-1}\right)$, $s_{n}(x)=s_{1}(T^{n-1}x)=1$ and $\epsilon_{n+1}(x)=\epsilon_{n}(x)$. Hence, $d_{n+1}(x)\geq d_{n}(x)+2$ when $\epsilon_{n+1}(x)=-\epsilon_{n}(x)$.\\
	\indent Suppose that $x=\frac{p}{q}\in(0,1)\cap\mathbb{Q}$, where $p,\ q\in\mathbb{N}$ and $p<q$. If $p=1$, then $Tx=0$. Otherwise, by the definition of $T$, there exists a positive integer $p_{1}$ that satisfies $1\leq p_{1}<p$ such that $Tx=\frac{p_{1}}{q}$. It follows that there exists a smallest positive integer $n$ such that $T^{n}x=0$. By \cref{eq:signed Engel expansion}, \cref{eq:finite signed Engel expansion} holds. In addition, if $x_{0}$ satisfies that $d_{k}(x_{0})$ is odd for some $k\in\mathbb{N}$, then there exists a positive integer $m$, such that $d_{k}(x_{0})=d_{1}(T^{k-1}x_{0})=2m+1$, that is, $T^{k-1}x_{0}=\frac{1}{2m+1}$. Hence, $T^{k}x_{0}=T(T^{k-1}x_{0})=0$, and \cref{eq:finity odd signed Engel expansion} holds by \cref{eq:signed Engel expansion}.\\
	\indent Now, assume that $x$ is irrational. According to the algorithm, we obtain an infinite series as shown on the right-hand side of \cref{eq:infinite signed Engel expansion}, which, since $d_{1}(x)\geq2$ and $d_{n}(x)$ is non-decreasing in $n$, converges to $x$. If the sequence $\left\{d_{n}(x)\right\}_{n\geq1}$ is bounded, then there exists a positive integers $N$ such that for all $n\geq N$, $d_{n}(x)=d_{N}(x)$. By $d_{n+1}(x)\geq d_{n}(x)+2$ when $\epsilon_{n+1}(x)=-\epsilon_{n}(x)$, we have $\epsilon_{n}(x)=\epsilon_{N}(x)$ for all $n\geq N$. At this point, it can be checked that $x$ is rational, which is a contradiction. So, $d_{n}(x)\to\infty$ as $n\to\infty$. On the other hand, assume that \cref{eq:infinite signed Engel expansion} holds for some $x\in(0,1)$. Based on the discussion of rational number expansions, it can be known that rational numbers can only be written as a finite sum. Hence, $x$ must be an irrational number.\\
	\indent Finally, we take the case of irrational numbers as an example to show the uniqueness. The case of rational numbers is similar. Let $x\in(0,1)\setminus\mathbb{Q}$ and define
	$$W_{n}(x)\coloneq\frac{\epsilon_{n+1}(x)}{d_{1}(x)\cdots d_{n}(x)d_{n+1}(x)}+\frac{\epsilon_{n+2}(x)}{d_{1}(x)\cdots d_{n}(x)d_{n+1}(x)d_{n+2}(x)}+\cdots.$$
	We first show that the positivity or negativity of $W_{n}(x)$ is determined by the first term, that is the sign of $\epsilon_{n+1}(x)$. In fact, based on the conditions satisfied by $d_{n}(x)$, we have
\begin{equation}\nonumber
	\begin{split}
		\epsilon_{n+1}(x)W_{n}(x) 
		&\geq\frac{1}{d_{1}(x)\cdots d_{n}(x)d_{n+1}(x)}-\frac{1}{d_{1}(x)\cdots d_{n}(x)d_{n+1}(x)(d_{n+1}(x)+2)}\\
		&\quad -\frac{1}{d_{1}(x)\cdots d_{n}(x)d_{n+1}(x)(d_{n+1}(x)+2)^{2}}-\cdots\\
	    &=\frac{1}{d_{1}(x)\cdots d_{n}(x)d_{n+1}(x)}-\frac{1}{d_{1}(x)\cdots d_{n}(x)d_{n+1}(x)(d_{n+1}(x)+1)}\\
	    &=\frac{1}{d_{1}(x)\cdots d_{n}(x)(d_{n+1}(x)+1)}>0.
	\end{split}
\end{equation}
Subsequently, we demonstrate that different infinite series on the right-hand side of \cref{eq:infinite signed Engel expansion} correspond to distinct irrational numbers. Suppose that
\begin{equation}\nonumber
	x=\frac{1}{d_{1}(x)}+\frac{\epsilon_{2}(x)}{d_{1}(x)d_{2}(x)}+\frac{\epsilon_{3}(x)}{d_{1}(x)d_{2}(x)d_{3}(x)}+\cdots,
\end{equation}
and
\begin{equation}\nonumber
	y=\frac{1}{d_{1}(y)}+\frac{\epsilon_{2}(y)}{d_{1}(y)d_{2}(y)}+\frac{\epsilon_{3}(y)}{d_{1}(y)d_{2}(y)d_{3}(y)}+\cdots.
\end{equation}
If $d_{1}(x)\neq d_{1}(y)$, we may assume that $d_{1}(y)>d_{1}(x)$. Then, we have $d_{1}(y)\geq d_{1}(x)+2$. It follows 
$$x>\frac{1}{d_{1}(x)}-\frac{1}{d_{1}(x)(d_{1}(x)+2)}-\frac{1}{d_{1}(x)(d_{1}(x)+2)^{2}}-\cdots=\frac{1}{d_{1}(x)+1},$$
and
$$y<\frac{1}{d_{1}(y)}+\frac{1}{(d_{1}(y))^{2}}+\frac{1}{(d_{1}(y))^{3}}+\cdots=\frac{1}{d_{1}(y)-1}.$$
Hence, $x\neq y$. If $d_{1}(x)=d_{1}(y)$, but $\epsilon_{2}(x)\neq\epsilon_{2}(y)$, then, we have
$$x-\frac{1}{d_{1}(x)}=\frac{\epsilon_{2}(x)}{d_{1}(x)d_{2}(x)}+\frac{\epsilon_{3}(x)}{d_{1}(x)d_{2}(x)d_{3}(x)}+\cdots,$$
and 
$$y-\frac{1}{d_{1}(y)}=\frac{\epsilon_{2}(y)}{d_{1}(y)d_{2}(y)}+\frac{\epsilon_{3}(y)}{d_{1}(y)d_{2}(y)d_{3}(y)}+\cdots.$$
So, the signs of $x-\frac{1}{d_{1}(x)}$ and $y-\frac{1}{d_{1}(y)}$ depend on the  $\epsilon_{2}(x)$ and $\epsilon_{2}(y)$, respectively. Hence, $x\neq y$. If  $d_{1}(x)=d_{1}(y)$, and $\epsilon_{2}(x)=\epsilon_{2}(y)$, then, we obtain
$$\epsilon_{2}(x)\left(xd_{1}(x)-1\right)=\frac{1}{d_{2}(x)}+\frac{\epsilon_{2}(x)\epsilon_{3}(x)}{d_{2}(x)d_{3}(x)}+\frac{\epsilon_{2}(x)\epsilon_{3}(x)\epsilon_{4}(x)}{d_{2}(x)d_{3}(x)d_{4}(x)}+\cdots,$$
and
$$\epsilon_{2}(y)\left(yd_{1}(y)-1\right)=\frac{1}{d_{2}(y)}+\frac{\epsilon_{2}(y)\epsilon_{3}(y)}{d_{2}(y)d_{3}(y)}+\frac{\epsilon_{2}(y)\epsilon_{3}(y)\epsilon_{4}(y)}{d_{2}(y)d_{3}(y)d_{4}(y)}+\cdots.$$
By repeating the above discussion, we can complete the proof of uniqueness.
\end{proof}
 The symbolic space related to the digit sequences of signed Engel expansion is constructed as follows. Let
\begin{equation}\nonumber
	\Sigma_{1}\coloneqq\left\{\left(\sigma_{1}\right)\colon\sigma_{1} \text{ is a positive integer and }\sigma_{1}\geq2 \right\},
\end{equation}
and for any $n\in\mathbb{N}\setminus\left\{1\right\}$,
\begin{equation}\nonumber
\Sigma_{n}\coloneqq\left\{\left(\sigma_{1},\delta_{2},\sigma_{2},\cdots,\delta_{n},\sigma_{n}\right)\colon
\begin{aligned}
	&\text{ }\delta_{i}\in\left\{1,-1\right\}\textnormal{ for all }2\leq i\leq n\\
	&\text{ }2\leq \sigma_{1}\leq \sigma_{2}\leq\cdots\leq \sigma_{n},\textnormal{ where }\sigma_{i}\textnormal{ is even for all}\ 1\leq i\leq n-1\\
    & \textnormal{ if }\delta_{i+1}=-\delta_{i}, \textnormal{ then }\sigma_{i+1}\geq\sigma_{i}+2\textnormal{ for all } 1\leq i\leq n-1
\end{aligned}
	\right\},
\end{equation}
where $\delta_{1}\coloneqq1$. Define
\begin{equation}\nonumber
	\Sigma_{\infty}\coloneqq\left\{\left(\sigma_{1},\delta_{2},\sigma_{2},\delta_{3},\sigma_{3},\cdots\right)\colon
	\begin{aligned}
		&\text{ }\delta_{i}\in\left\{1,-1\right\}\textnormal{ for all }i\geq 2\\
		&\text{ }2\leq \sigma_{1}\leq \sigma_{2}\leq\sigma_{3}\cdots,\textnormal{ where }\sigma_{i}\text{ is even}\textnormal{ for all }i\in\mathbb{N}\\
		&  \textnormal{ if }\delta_{i+1}=-\delta_{i}, \textnormal{ then }\sigma_{i+1}\geq\sigma_{i}+2\textnormal{ for all }i\in\mathbb{N}
	\end{aligned}
	\right\}.
\end{equation}
By \cref{pro:signed Engel Expansion}, for all $n\in\mathbb{N}$, we also define
\begin{equation}\nonumber
	\Sigma_{n}^{'}\coloneqq\left\{\left(\sigma_{1},\cdots,\delta_{n},\sigma_{n}\right)\colon\left(\sigma_{1},\cdots,\delta_{n},\sigma_{n}\right)\in\Sigma_{n}\textnormal{ and }\sigma_{n}\textnormal{ is even }\right\}.
\end{equation}
\begin{definition}\label{def:admissible sequences}
	We say that a finite sequence $\left(\sigma_{1},\delta_{2},\sigma_{2},\cdots,\delta_{n},\sigma_{n}\right)$ for some $n\in\mathbb{N}$ is signed Engel admissible if there exists $x\in(0,1)$ such that $d_{i}(x)=\sigma_{i}$ and $\epsilon_{i}(x)=\delta_{i}$ for all $1\leq i\leq n$. An infinite sequence $\left(\sigma_{1},\delta_{2},\sigma_{2},\delta_{3},\sigma_{3},\cdots\right)$ is said to be signed Engel admissible if there exists $x\in(0,1)$ such that $d_{n}(x)=\sigma_{n}$ and $\epsilon_{n}(x)=\delta_{n}$ for all $ n\in\mathbb{N}$. 
\end{definition}
We denote by $\Sigma_{ad}$ the set of all signed Engel admissible sequences. According to \cref{pro:signed Engel Expansion}, we have $\Sigma_{ad}=\bigcup_{n\in\mathbb{N}}\Sigma_{n}\cup\Sigma_{\infty}$ .
\begin{definition}
	For any $n\in\mathbb{N}$ and any $\left(\sigma_{1},\delta_{2},\sigma_{2},\cdots,\delta_{n},\sigma_{n}\right)\in\Sigma_{n}$, we call 
	$$I_{n}\left(\sigma_{1},\delta_{2},\sigma_{2},\cdots,\delta_{n},\sigma_{n}\right)=\left\{x\in(0,1):d_{1}(x)=\sigma_{1},\epsilon_{2}(x)=\delta_{2},d_{2}(x)=\sigma_{2},\cdots,\epsilon_{n}(x)=\delta_{n},d_{n}(x)=\sigma_{n}\right\}$$ 
	a basic interval of order $n$ related to signed Engel expansion.
\end{definition}
The following proposition presents some useful descriptions of the basic intervals of signed Engel expansion.
\begin{proposition}\label{specific form of the basic interval}
	Let $\left(\sigma_{1},\delta_{2},\sigma_{2},\cdots,\delta_{n},\sigma_{n}\right)\in\Sigma_{n}$. If $\sigma_{n}$ is odd, then $I_{n}\left(\sigma_{1},\delta_{2},\sigma_{2},\cdots,\delta_{n},\sigma_{n}\right)$ is a singleton. Otherwise, $I_{n}\left(\sigma_{1},\delta_{2},\sigma_{2},\cdots,\delta_{n},\sigma_{n}\right)$ is an open interval of positive length with two endpoints
	$$x_{1}=\frac{1}{\sigma_{1}}+\frac{\delta_{2}}{\sigma_{1}\sigma_{2}}+\cdots+\frac{\delta_{n-1}}{\sigma_{1}\sigma_{2}\cdots \sigma_{n-1}}+\frac{\delta_{n}}{\sigma_{1}\sigma_{2}\cdots \sigma_{n-1}\left(\sigma_{n}-1\right)},$$
	and
	$$x_{2}=\frac{1}{\sigma_{1}}+\frac{\delta_{2}}{\sigma_{1}\sigma_{2}}+\cdots+\frac{\delta_{n-1}}{\sigma_{1}\sigma_{2}\cdots \sigma_{n-1}}+\frac{\delta_{n}}{\sigma_{1}\sigma_{2}\cdots \sigma_{n-1}\left(\sigma_{n}+1\right)}.$$
	More precisely, if $\sigma_{n}$ is even, then
	\begin{equation}\nonumber
		\begin{split}
				I_{n}&\left(\sigma_{1},\delta_{2},\sigma_{2},\cdots,\delta_{n},\sigma_{n}\right)\\
				&=\begin{cases}
				\left(\frac{1}{\sigma_{1}}+\cdots+\frac{\delta_{n-1}}{\sigma_{1}\cdots\sigma_{n-1}}+\frac{\delta_{n}}{\sigma_{1}\cdots\sigma_{n-1}\left(\sigma_{n}-1\right)},\ \frac{1}{\sigma_{1}}+\cdots+\frac{\delta_{n-1}}{\sigma_{1}\cdots\sigma_{n-1}}+\frac{\delta_{n}}{\sigma_{1}\cdots\sigma_{n-1}\left(\sigma_{n}+1\right)}\right),&\text{if}\ \delta_{n}=-1,\\
				\left(\frac{1}{\sigma_{1}}+\cdots+\frac{\delta_{n-1}}{\sigma_{1}\cdots\sigma_{n-1}}+\frac{\delta_{n}}{\sigma_{1}\cdots\sigma_{n-1}\left(\sigma_{n}+1\right)},\ \frac{1}{\sigma_{1}}+\cdots+\frac{\delta_{n-1}}{\sigma_{1}\cdots\sigma_{n-1}}+\frac{\delta_{n}}{\sigma_{1}\cdots\sigma_{n-1}\left(\sigma_{n}-1\right)}\right),&\text{if}\ \delta_{n}=1.
			\end{cases}
		\end{split}
	\end{equation}
	Hence,
	\begin{equation}\label{the length of basic interval}
		|I_{n}\left(\sigma_{1},\delta_{2},\sigma_{2},\cdots,\delta_{n},\sigma_{n}\right)|=
		\begin{cases}
			0,&\ \text{if}\ \sigma_{n}\ \text{is odd},\\
			\frac{2}{\sigma_{1}\sigma_{2}\cdots \sigma_{n-1}\left(\sigma_{n}-1\right)\left(\sigma_{n}+1\right)},&\ \text{if}\ \sigma_{n}\ \text{is even}.
		\end{cases}
	\end{equation}
\end{proposition}
\begin{proof}
	The case $n=1$ is trivial, and we only consider the case $n\geq2$.\\
	\indent If $\sigma_{n}$ is odd, then, by \cref{pro:signed Engel Expansion}, $I_{n}\left(\sigma_{1},\delta_{2},\sigma_{2},\cdots,\delta_{n},\sigma_{n}\right)$ contains precisely the rational point
	$$\frac{1}{\sigma_{1}}+\frac{\delta_{2}}{\sigma_{1}\sigma_{2}}+\cdots+\frac{\delta_{n}}{\sigma_{1}\sigma_{2}\cdots \sigma_{n-1}\sigma_{n}}.$$
	So, $|I_{n}\left(\sigma_{1},\delta_{2},\sigma_{2},\cdots,\delta_{n},\sigma_{n}\right)|=0$.\\
	\indent Now, suppose that $\sigma_{n}$ is even. If $\delta_{n}=-1$, then $x\in I_{n}\left(\sigma_{1},\delta_{2},\sigma_{2},\cdots,\delta_{n},\sigma_{n}\right)$ if and only if
	\begin{equation}\nonumber
		\begin{split}
			x&>\frac{1}{\sigma_{1}}+\cdots+\frac{\delta_{n-1}}{\sigma_{1}\cdots \sigma_{n-1}}+\frac{\delta_{n}}{\sigma_{1}\cdots \sigma_{n-1}\sigma_{n}}+\frac{-1}{\sigma_{1}\cdots \sigma_{n-1}\left(\sigma_{n}\right)^{2}}+\frac{-1}{\sigma_{1}\cdots \sigma_{n-1}\left(\sigma_{n}\right)^{3}}+\cdots\\
			&=\frac{1}{\sigma_{1}}+\cdots+\frac{\delta_{n-1}}{\sigma_{1}\cdots \sigma_{n-1}}+\frac{\delta_{n}}{\sigma_{1}\cdots \sigma_{n-1}\left(\sigma_{n}-1\right)},
		\end{split}
	\end{equation}
	and 
	\begin{equation}\nonumber
		\begin{split}
			x&<\frac{1}{\sigma_{1}}+\cdots+\frac{\delta_{n-1}}{\sigma_{1}\cdots \sigma_{n-1}}+\frac{\delta_{n}}{\sigma_{1}\cdots \sigma_{n-1}\sigma_{n}}+\frac{1}{\sigma_{1}\cdots \sigma_{n-1}\sigma_{n}\left(\sigma_{n}+2\right)}+\frac{1}{\sigma_{1}\cdots \sigma_{n-1}\sigma_{n}\left(\sigma_{n}+2\right)^{2}}+\cdots\\
			&=\frac{1}{\sigma_{1}}+\cdots+\frac{\delta_{n-1}}{\sigma_{1}\cdots \sigma_{n-1}}+\frac{\delta_{n}}{\sigma_{1}\cdots \sigma_{n-1}\left(\sigma_{n}+1\right)}.
		\end{split}
	\end{equation}
	Hence,
	\begin{equation}\nonumber
		\begin{split}
			I_{n}&\left(\sigma_{1},\delta_{2},\sigma_{2},\cdots,\delta_{n},\sigma_{n}\right)\\
			&=\left(\frac{1}{\sigma_{1}}+\cdots+\frac{\delta_{n-1}}{\sigma_{1}\cdots\sigma_{n-1}}+\frac{\delta_{n}}{\sigma_{1}\cdots\sigma_{n-1}\left(\sigma_{n}-1\right)},\ \frac{1}{\sigma_{1}}+\cdots+\frac{\delta_{n-1}}{\sigma_{1}\cdots\sigma_{n-1}}+\frac{\delta_{n}}{\sigma_{1}\cdots\sigma_{n-1}\left(\sigma_{n}+1\right)}\right).
		\end{split}
	\end{equation}
	The case $\delta_{n}=1$ is analogous to that for $\delta_{n}=-1$, we omit the details here.
\end{proof}
We conclude this part with a conditional probability formula. We use Lebesgue measure as the underlying probability over Borel sets in $(0,1)$. Then, the digit sequence $\left\{d_{n}\right\}_{n\geq1}$ of signed Engel expansion can be regarded as a sequence of random variables on the probability space $\left((0,1),\mathcal{B},\mathcal{L}\right)$, where $\mathcal{B}$ is the Borel $\sigma$-algebra on $(0,1)$.
\begin{proposition}\label{dn homo markov chain}
	The digit sequence $\left\{d_{n}\right\}_{n\geq1}$ forms a time-homogeneous Markov chain with initial distribution
	\begin{equation}\label{initial distribution}
		\mathcal{L}\left(d_{1}=2k\right)=\frac{2}{\left(2k-1\right)\left(2k+1\right)},
	\end{equation}
	and one step transition probabilities
	\begin{equation}\label{transition probabilities}
		\mathcal{L}\left(d_{n+1}=2l\mid d_{n}=2k\right)=
		\begin{cases}
			\frac{1}{2k},&\ \textnormal{if}\ l=k,\\
			\frac{\left(2k-1\right)\left(2k+1\right)}{k\left(2l-1\right)\left(2l+1\right)},&\ \textnormal{if}\ l\geq k+1,
		\end{cases}
	\end{equation}
	where $l,\ k\in\mathbb{N}$, and $l\geq k$. 
\end{proposition}
\begin{remark}
	 For notational brevity, we often omit the generic element of a set when referring to its Lebesgue measure. For example, $\mathcal{L}(d_{1}=2k)$ and $\mathcal{L}(d_{n+1}=2l\mid d_{n}=2k)$ denote, respectively, $\mathcal{L}(\{x\in\mathbb{I}\colon d_{1}(x)=2k\})$ and $\mathcal{L}(\{x\in\mathbb{I}\colon d_{n+1}(x)=2l\mid d_{n}(x)=2k\})$.
\end{remark}
\begin{proof}
	 For any $\left(\sigma_{1},\cdots,\delta_{n-1},\sigma_{n-1},\delta_{n},2k\right)\in\Sigma_{n}$,
	 \begin{equation}\nonumber
	 	\begin{split}
	 		&\quad\mathcal{L}\left(d_{n+1}=2l\mid d_{1}=\sigma_{1},\cdots,\epsilon_{n-1}=\delta_{n-1},d_{n-1}=\sigma_{n-1},\epsilon_{n}=\delta_{n},d_{n}=2k\right)\\
	 		&=\frac{\mathcal{L}\left( d_{1}=\sigma_{1},\cdots,\epsilon_{n-1}=\delta_{n-1},d_{n-1}=\sigma_{n-1},\epsilon_{n}=\delta_{n},d_{n}=2k,d_{n+1}=2l\right)}{\mathcal{L}\left( d_{1}=\sigma_{1},\cdots,\epsilon_{n-1}=\delta_{n-1},d_{n-1}=\sigma_{n-1},\epsilon_{n}=\delta_{n},d_{n}=2k\right)}
	 	\end{split}
	 \end{equation}
	 First, by \cref{specific form of the basic interval}, we have
	 \begin{equation}\nonumber
	 	\begin{split}
	 		&\quad\mathcal{L}\left( d_{1}=\sigma_{1},\cdots,\epsilon_{n-1}=\delta_{n-1},d_{n-1}=\sigma_{n-1},\epsilon_{n}=\delta_{n},d_{n}=2k\right)\\
	 		&=\frac{2}{\sigma_{1}\cdots\sigma_{n-1}\left(2k-1\right)\left(2k+1\right)}
	 	\end{split}
	 \end{equation}
	 If $l=k$, then 
	 \begin{equation}\nonumber
	 	\begin{split}
	 		&\quad\mathcal{L}\left(d_{1}=\sigma_{1},\cdots,\epsilon_{n-1}=\delta_{n-1},d_{n-1}=\sigma_{n-1},\epsilon_{n}=\delta_{n},d_{n}=2k,d_{n+1}=2l\right)\\
	 		&=\mathcal{L}\left(d_{1}=\sigma_{1},\cdots,\epsilon_{n-1}=\delta_{n-1},d_{n-1}=\sigma_{n-1},\epsilon_{n}=\delta_{n},d_{n}=2k,\epsilon_{n+1}=1,d_{n+1}=2l\right)\\
	 		&=\frac{2}{\sigma_{1}\cdots\sigma_{n-1}\left(2k\right)\left(2l-1\right)\left(2l+1\right)}=\frac{2}{\sigma_{1}\cdots\sigma_{n-1}\left(2k\right)\left(2k-1\right)\left(2k+1\right)}.
	 	\end{split}
	 \end{equation}
	 If $l>k$, then
	 \begin{equation}\nonumber
	 	\begin{split}
	 		&\quad\mathcal{L}\left(d_{1}=\sigma_{1},\cdots,\epsilon_{n-1}=\delta_{n-1},d_{n-1}=\sigma_{n-1},\epsilon_{n}=\delta_{n},d_{n}=2k,d_{n+1}=2l\right)\\
	 		&=\mathcal{L}\left(d_{1}=\sigma_{1},\cdots,\epsilon_{n-1}=\delta_{n-1},d_{n-1}=\sigma_{n-1},\epsilon_{n}=\delta_{n},d_{n}=2k,\epsilon_{n+1}=1,d_{n+1}=2l\right)\\
	 		&\quad+\mathcal{L}\left(d_{1}=\sigma_{1},\cdots,\epsilon_{n-1}=\delta_{n-1},d_{n-1}=\sigma_{n-1},\epsilon_{n}=\delta_{n},d_{n}=2k,\epsilon_{n+1}=-1,d_{n+1}=2l\right)\\
	 		&=\frac{4}{\sigma_{1}\cdots\sigma_{n-1}\left(2k\right)\left(2l-1\right)\left(2l+1\right)}.
	 	\end{split}
	 \end{equation}
	 Hence, we have
	 \begin{equation}\nonumber
	 	\begin{split}
	 		&\quad\mathcal{L}\left(d_{n+1}=2l\mid d_{1}=\sigma_{1},\cdots,\epsilon_{n-1}=\delta_{n-1},d_{n-1}=\sigma_{n-1},\epsilon_{n}=\delta_{n},d_{n}=2k\right)\\
	 		&=
	 		\begin{cases}
	 			\frac{1}{2k},\quad &\text{if}\ l=k,\\
	 			\frac{\left(2k-1\right)\left(2k+1\right)}{k\left(2l-1\right)\left(2k+1\right)},\quad &\text{if}\ l\geq k+1,
	 		\end{cases}
	 	\end{split}
	 \end{equation}
	 which implies that the sequence $\left\{d_{n}\right\}_{n\geq1}$ forms a time-homogeneous Markov chain. \cref{initial distribution} follows directly from \cref{specific form of the basic interval}.\\
	 \indent Now, we prove \cref{transition probabilities}.
   	By \cref{specific form of the basic interval}, we have
	\begin{equation}\nonumber
		\begin{split}
			\mathcal{L}\left(d_{n}=2k\right)
			&=\mathcal{L}\left(\bigcup_{\left(\sigma_{1},\cdots,\delta_{n-1},\sigma_{n-1},\delta_{n},2k\right)\in\Sigma_{n}}I_{\left(\sigma_{1},\cdots,\delta_{n-1},\sigma_{n-1},\delta_{n},2k\right)}\right)\\
			&=\frac{2}{\left(2k-1\right)\left(2k+1\right)}\sum_{\left(\sigma_{1},\cdots,\delta_{n-1},\sigma_{n-1},\delta_{n},2k\right)\in\Sigma_{n}}\frac{1}{\sigma_{1}\cdots \sigma_{n-1}},
		\end{split}
	\end{equation}
	and
	\begin{equation}\nonumber
		\begin{split}
			\mathcal{L}\left(d_{n}=2k,d_{n+1}=2l\right)
			&=\mathcal{L}\left(\bigcup_{\left(\sigma_{1},\cdots,\delta_{n-1},\sigma_{n-1},\delta_{n},2k,\delta_{n+1},2l\right)\in\Sigma_{n+1}}I_{\left(\sigma_{1},\cdots,\delta_{n-1},\sigma_{n-1},\delta_{n},2k,\delta_{n+1},2l\right)}\right)\\
			&=\frac{1}{k\left(2l-1\right)\left(2l+1\right)}\sum_{\left(\sigma_{1},\cdots,\delta_{n-1},\sigma_{n-1},\delta_{n},2k,\delta_{n+1},2l\right)\in\Sigma_{n+1}}\frac{1}{\sigma_{1}\cdots \sigma_{n-1}}.\\
		\end{split}
	\end{equation}
	If $l=k$, then $\delta_{n+1}=1$. It follows that
	\begin{equation}\nonumber
		\mathcal{L}\left(d_{n}=2k,d_{n+1}=2l\right)=\frac{1}{k\left(2k-1\right)\left(2k+1\right)}\sum_{\left(\sigma_{1},\cdots,\delta_{n-1},\sigma_{n-1},\delta_{n},2k\right)\in\Sigma_{n}}\frac{1}{\sigma_{1}\cdots \sigma_{n-1}}.
	\end{equation}
	Hence, $\mathcal{L}\left(d_{n+1}=2k\mid d_{n}=2k\right)=\frac{1}{2k}$. If $l>k$, then $\delta_{n+1}$ can be either $1$ or $-1$. Noting that $$I_{\left(\sigma_{1},\cdots,\delta_{n-1},\sigma_{n-1},\delta_{n},2k,1,2l\right)}=I_{\left(\sigma_{1},\cdots,\delta_{n-1},\sigma_{n-1},\delta_{n},2k,-1,2l\right)},$$ we get
	\begin{equation}\nonumber
		\mathcal{L}\left(d_{n}=2k,d_{n+1}=2l\right)=\frac{2}{k\left(2l-1\right)\left(2l+1\right)}\sum_{\left(\sigma_{1},\cdots,\delta_{n-1},\sigma_{n-1},\delta_{n},2k\right)\in\Sigma_{n}}\frac{1}{\sigma_{1}\cdots \sigma_{n-1}}.
	\end{equation}
	Hence,
	\begin{equation}\nonumber
		\mathcal{L}\left(d_{n+1}=2l\mid d_{n}=2k\right)=\frac{\left(2k-1\right)\left(2k+1\right)}{k\left(2l-1\right)\left(2l+1\right)}.
	\end{equation}
    The proof is completed.
\end{proof}
The Borel--Cantelli lemma is frequently used in our proof, and we state it here.
\begin{lemma}\label{Borel--Cantelli lemma}(\cite[Theorems 2.3.1 and 2.3.7]{MR3930614})
	Let $(\Omega,\mathcal{F},\mathbb{P})$ be a probability space and $\{E_{n}\}_{n\geq1}$ be a sequence of events. The following results hold.
	\begin{enumerate}[(1)]
		\item If $\sum_{n=1}^{\infty}\mathbb{P}(E_{n})<\infty$, then $\mathbb{P}(E_{n}\ \textnormal{i.m.}\ n)=0$.
		\item  If $\sum_{n=1}^{\infty}\mathbb{P}(E_{n})=\infty$ and the events $E_{n}$ are independent, then $\mathbb{P}(E_{n}\ \textnormal{i.m.}\ n)=1$.
	\end{enumerate}
\end{lemma}
\section{Proofs of the results}\label{proof of the results}
\subsection{Proofs of Limit Theorems}\label{proof of the limit theorems}
Since the distributional properties of a Markov chain are uniquely determined by its initial distribution and transition probabilities, one can study the metric theory of the sequence $\left\{d_{n}\right\}_{n\geq1}$ through any Markov chain to which \cref{initial distribution,transition probabilities} apply. Based on \cref{initial distribution,transition probabilities}, we introduce the following  Markov chain $\left\{D_{n}\right\}_{n\geq1}$ as a substitute for the sequence $\left\{d_{n}\right\}_{n\geq1}$. For this purpose, we first define an even function. For any $t\in[1,\infty)$, let
\begin{equation}\nonumber
	[t]_{E}\coloneq2k,\ \ \textnormal{if }\ 2k-1\leq t<2k+1,\ k\in\mathbb{N}.
\end{equation}
\indent On the probability space $\left((0,1),\mathcal{B},\mathcal{L}\right)$, let $\left\{X_{i}\right\}_{i\geq1}$ be a sequence of independent and identically distributed non-negative random variables, each following an exponential distribution with a rate of $1$. For any $x\in(0,1)$ and $n\in\mathbb{N}$, define 
\begin{equation}\nonumber
	D_{1}(x)\coloneq\left[\exp\left(X_{1}\left(x\right)\right)\right]_{E},\quad\textnormal{ and }\quad D_{n+1}(x)\coloneq\left[\frac{\left(D_{n}(x)-1\right)\left(D_{n}(x)+1\right)}{D_{n}(x)}\exp\left(X_{n+1}\left(x\right)\right)\right]_{E}.
\end{equation}
The following lemma implies that $\left\{D_{n}\right\}_{n\geq1}$ is indeed a Markov chain.
\begin{lemma}\label{Dn homo Markov chain}
	The sequence $\left\{D_{n}\right\}_{n\geq1}$ forms a time-homogeneous Markov chain with its initial distribution and one step transition probabilities satisfying \cref{initial distribution} and \cref{transition probabilities}, respectively.
\end{lemma}
\begin{proof}
	Let $k,\ l\in\mathbb{N}$ and $l\geq k$. First, we have 
	\begin{equation}\nonumber
		\mathcal{L}\left(D_{1}=2k\right)=\mathcal{L}\left(\left[\exp X_{1}\right]_{E}=2k\right)=\mathcal{L}\left(2k-1\leq\exp X_{1}<2k+1\right)=\frac{2}{\left(2k-1\right)\left(2k+1\right)}.
	\end{equation}
	For each $n\in\mathbb{N}$, since $D_{n}$ depends on $X_{1},\ X_{2},\cdots,\ X_{n}$ and is independent of $X_{n+1}$, it follows that
	\begin{equation}\nonumber
		\begin{split}
			\mathcal{L}\left(D_{n+1}=2l\mid D_{n}=2k\right)
			&=\mathcal{L}\left(\left[\frac{\left(2k-1\right)\left(2k+1\right)}{2k}\exp X_{n+1}\right]_{E}=2l\mid D_{n}=2k\right)\\
			&=\mathcal{L}\left(\left[\frac{\left(2k-1\right)\left(2k+1\right)}{2k}\exp X_{n+1}\right]_{E}=2l\right).
		\end{split}
	\end{equation}
	When $l=k$, 
	\begin{equation}\nonumber
		\mathcal{L}\left(D_{n+1}=2k\mid D_{n}=2k\right)=\mathcal{L}\left(\frac{2k}{2k+1}\leq\exp X_{n+1}<\frac{2k}{2k-1}\right)=\mathcal{L}\left(1\leq\exp X_{n+1}<\frac{2k}{2k-1}\right)=\frac{1}{2k}.
	\end{equation}
	Otherwise, that is, $l>k$, we get
	\begin{equation}\nonumber
		\mathcal{L}\left(D_{n+1}=2l\mid D_{n}=2k\right)=\mathcal{L}\left(\frac{2k\left(2l-1\right)}{\left(2k-1\right)\left(2k+1\right)}\leq\exp X_{n+1}<\frac{2k\left(2l+1\right)}{\left(2k-1\right)\left(2k+1\right)}\right)=\frac{\left(2k-1\right)\left(2k+1\right)}{k\left(2l-1\right)\left(2l+1\right)}.
	\end{equation}
\end{proof}
The lemma below states that for almost all $x\in(0,1)$, the equality $D_{n}(x)=D_{n+1}(x)$ holds for only finitely many $n$. Similarly, the digit sequence $\left\{d_{n}\right\}_{n\geq1}$ of signed Engel expansion shares the same property.
\begin{lemma}\label{dn=dn+1 for finitely many n}
	We have
	\begin{equation}\nonumber
		\mathcal{L}\left(\left\{x\in\left(0,1\right)\colon D_{n+1}(x)=D_{n}(x)~~\textnormal{i.m.}~n\right\}\right)=0.
	\end{equation}
\end{lemma}
\begin{proof}
	For each $n\geq2$, by the law of total probability and \cref{Dn homo Markov chain}, we have
	\begin{equation}\nonumber
		\begin{split}
			\mathcal{L}\left(\left\{x\in\left(0,1\right)\colon D_{n+1}(x)=D_{n}(x)\right\}\right)
			&=\sum_{k=1}^{\infty}\mathcal{L}\left(D_{n+1}=D_{n}\mid D_{n}=2k\right)\cdot\mathcal{L}\left(D_{n}=2k\right)=\sum_{k=1}^{\infty}\frac{\mathcal{L}\left(D_{n}=2k\right)}{2k}\\
			&=\sum_{k=1}^{\infty}\sum_{j=1}^{k}\frac{\mathcal{L}\left(D_{n}=2k\mid D_{n-1}=2j\right)}{2k}\cdot\mathcal{L}\left(D_{n-1}=2j\right)\\
			&=\sum_{j=1}^{\infty}\frac{\mathcal{L}\left(D_{n-1}=2j\right)}{2j}\sum_{k=j}^{\infty}\frac{j}{k}\mathcal{L}\left(D_{n}=2k\mid D_{n-1}=2j\right).
		\end{split}
	\end{equation}
	By \cref{Dn homo Markov chain}, we obtain
	\begin{equation}\nonumber
		\begin{split}
			\sum_{k=j}^{\infty}\frac{j}{k}\mathcal{L}\left(D_{n}=2k\mid D_{n-1}=2j\right)
			&=\frac{1}{2j}+\sum_{k=j+1}^{\infty}\frac{\left(2j-1\right)\left(2j+1\right)}{k\left(2k-1\right)\left(2k+1\right)}\\
			&\leq\frac{1}{2j}+\frac{2j-1}{\left(j+1\right)\left(2j+3\right)}+\sum_{k=j+2}^{\infty}\frac{2\left(2j-1\right)\left(2j+1\right)}{\left(2k-3\right)\left(2k-1\right)\left(2k+1\right)}\\
			&=\frac{1}{2j}+\frac{2j-1}{\left(j+1\right)\left(2j+3\right)}+\frac{2j-1}{2\left(2j+3\right)}.
		\end{split}
	\end{equation}
	Then $\sum_{k=1}^{\infty}\frac{1}{k}\mathcal{L}\left(D_{n}=2k\mid D_{n-1}=2\right)\leq\frac{7}{10}$, $\sum_{k=2}^{\infty}\frac{2}{k}\mathcal{L}\left(D_{n}=2k\mid D_{n-1}=4\right)\leq\frac{17}{28}$, and for any $j\geq3$, $\sum_{k=j}^{\infty}\frac{j}{k}\mathcal{L}\left(D_{n}=2k\mid D_{n-1}=2j\right)\leq\frac{11}{12}$. So, by induction on $n$, we have
	\begin{equation}\nonumber
		\begin{split}
			\mathcal{L}\left(\left\{x\in\left(0,1\right)\colon D_{n+1}(x)=D_{n}(x)\right\}\right)
			&=\sum_{k=1}^{\infty}\frac{\mathcal{L}\left(D_{n}=2k\right)}{2k}\leq\frac{11}{12}\sum_{j=1}^{\infty}\frac{\mathcal{L}\left(D_{n-1}=2j\right)}{2j}\\
			&\leq\cdots\leq\left(\frac{11}{12}\right)^{n-1}\sum_{j=1}^{\infty}\frac{\mathcal{L}\left(D_{1}=2j\right)}{2j}.
		\end{split}
	\end{equation}
	It follows that 
	\begin{equation}\nonumber
		\sum_{n=1}^{\infty}\mathcal{L}\left(\left\{x\in\left(0,1\right)\colon D_{n+1}(x)=D_{n}(x)\right\}\right)<\infty.
	\end{equation} 
	By \cref{Borel--Cantelli lemma}, we obtain the desired result.
\end{proof}
\begin{proof}[Proof of \cref{limit theorems}]
	For any $n\in\mathbb{N}$, let $S_{n}\coloneq\sum_{k=1}^{n}X_{k}$. According to  \cite[Theorems 2.4.1, 3.4.1, and 8.5.2]{MR3930614}, we can obtain the corresponding limit theorems by replacing $\log d_{n}(x)$ with $S_{n}(x)$ in \cref{limit theorems}.\\
	\indent By the definition, for any $x\in(0,1)$ and $n\geq2$, we have 
	\begin{equation}\label{Dn control Dn-1}
		D_{n-1}(x)\left(1-\frac{1}{D_{n-1}^{2}(x)}\right)\exp\left(X_{n}(x)\right)-1\leq D_{n}(x)<D_{n-1}(x)\exp\left(X_{n}(x)\right)+1.
	\end{equation}
	Based on \cref{dn=dn+1 for finitely many n} and the fact that $D_{n}(x)$ is even, for almost all $x\in(0,1)$, for sufficiently large $n$, we have
	\begin{equation}\nonumber
		D_{n-1}(x)\left(1-\frac{1}{n}\right)\exp\left(X_{n}(x)\right)\leq D_{n}(x)<D_{n-1}(x)\left(1+\frac{1}{n}\right)\exp\left(X_{n}(x)\right),
	\end{equation}
	that is,
	\begin{equation}\nonumber
		X_{n}(x)+\log\left(1-\frac{1}{n}\right)\leq \log D_{n}(x)-\log D_{n-1}(x)<X_{n}(x)+\log\left(1+\frac{1}{n}\right).
	\end{equation}
	A simple summation shows that $\log D_{n}(x)-S_{n}(x)=o(\sqrt{n})$, which implies that $\log D_{n}(x)$ and $S_{n}(x)$ share the same limit theorems. The proof is completed.
\end{proof}
The proof of limit theorems for the gap sequence $\left\{\Delta_{n}\right\}_{n\geq1}$ needs the following result. We state it here and omit its proof, which can be found in \cite{MR4954368}.
\begin{lemma}\label{lim of Xn divided square root of n}(\cite[Lemma 3.4]{MR4954368})
	For Lebesgue almost all $x\in(0,1)$,
	\begin{equation}\nonumber
		\lim_{n\to\infty}\frac{\log\left(\exp\left(X_{n}(x)-1\right)\right)}{\sqrt{n}}=0\quad\textnormal{and}\quad\lim_{n\to\infty}\frac{X_{n}(x)}{\sqrt{n}}=0.
	\end{equation}
\end{lemma}
\begin{proof}[Proof of \cref{limit theorems for gap sequences}]
	By \cref{Dn control Dn-1}, for $n\geq2$, we have
	\begin{equation}\nonumber
	D_{n-1}(x)\left(\left(1-\frac{1}{D_{n-1}^{2}(x)}\right)\exp\left(X_{n}(x)\right)-\frac{1}{D_{n-1}(x)}-1\right)\leq D_{n}(x)-D_{n-1}(x)<
		D_{n-1}(x)\exp\left(X_{n}(x)\right).
	\end{equation}
	In addition, by \cref{dn=dn+1 for finitely many n}, for almost all $x\in(0,1)$, $D_{n}(x)-D_{n-1}(x)\geq2$ holds for sufficiently large $n$. Then,
	\begin{equation}\nonumber
		\begin{split}
			D_{n}(x)-D_{n-1}(x)
			&\geq D_{n-1}(x)\cdot\max\left\{\frac{2}{D_{n-1}(x)},\ \left(1-\frac{1}{D_{n-1}^{2}(x)}\right)\exp\left(X_{n}(x)\right)-\frac{1}{D_{n-1}(x)}-1\right\}\\
			&\geq\frac{1}{2}D_{n-1}(x)\left(\left(1-\frac{1}{D_{n-1}^{2}(x)}\right)\exp\left(X_{n}(x)\right)+\frac{1}{D_{n-1}(x)}-1\right)\\
			&\geq\frac{1}{2}D_{n-1}(x)\left(1-\frac{1}{D_{n-1}(x)}\right)\left(\exp\left(X_{n}(x)\right)-1\right)\geq\frac{1}{4}D_{n-1}(x)\left(\exp\left(X_{n}(x)\right)-1\right).
		\end{split}
	\end{equation}
	Hence, for almost all $x\in(0,1)$, for sufficiently large $n$, we have
	\begin{equation}\nonumber
		\log\frac{1}{4}+\log\left(\exp\left(X_{n}(x)\right)-1\right)\leq\log\left(D_{n}(x)-D_{n-1}(x)\right)-\log D_{n-1}(x)<X_{n}(x).
	\end{equation}
	By \cref{lim of Xn divided square root of n}, we get $\log\left(D_{n}(x)-D_{n-1}(x)\right)-\log D_{n}(x)=o(\sqrt{n})$, which implies that $\log D_{n}(x)$ and $\log\left(D_{n}(x)-D_{n-1}(x)\right)$  share the same limit theorems. The proof is completed.
\end{proof}
\subsection{Proof of Borel--Bernstein theorem}\label{Proof of Borel--Bernstein theorem}
\indent We recall that, for all $x\in\mathbb{I}$, the signed Engel expansion of $x$ is given by
\begin{equation}\nonumber
	x=\frac{1}{d_{1}(x)}+\frac{\epsilon_{2}(x)}{d_{1}(x)d_{2}(x)}+\frac{\epsilon_{3}(x)}{d_{1}(x)d_{2}(x)d_{3}(x)}+\cdots.
\end{equation}
For $n\in\mathbb{N}$, it can be checked that
\begin{equation}\nonumber
	T^{n-1}x=\frac{1}{d_{n}(x)}+\frac{s_{n}(x)}{d_{n}(x)d_{n+1}(x)}+\frac{s_{n}(x)s_{n+1}(x)}{d_{n}(x)d_{n+1}(x)d_{n+2}(x)}+\cdots.
\end{equation}
Then, we have $\frac{1}{d_{n}(x)+1}<T^{n-1}x<\frac{1}{d_{n}(x)-1}$. 
Define
\begin{equation}\nonumber
	y_{n}(x)\coloneqq
	\begin{cases}
		x, &\text{ if }n=1,\\
		\left(d_{n-1}(x)-1\right)T^{n-1}x, &\text{ if }n\geq2 \text{ and } \epsilon_{n}(x)=\epsilon_{n-1}(x),\\
		\left(d_{n-1}(x)+1\right)T^{n-1}x, &\text{ if }n\geq2 \text{ and } \epsilon_{n}(x)=-\epsilon_{n-1}(x).
	\end{cases}
\end{equation}

The following lemma shows that the random variable $y_{n}$ follows a uniform distribution on $(0,1)$.
\begin{lemma}\label{the probability of yn}
	Let $0\leq c\leq1$. For any $n\in\mathbb{N}$, we have $\mathcal{L}\left(y_{n}\leq c\right)=c$.
\end{lemma}
\begin{proof}
	When $n=1$, we have $\mathcal{L}\left(y_{1}\leq c\right)=\mathcal{L}\left(x\leq c\right)=c$. Now, suppose that $n\geq2$. For any $\left(\sigma_{1},\delta_{2},\sigma_{2},\cdots,\delta_{n-1},\sigma_{n-1}\right)\in\Sigma_{n-1}^{'}$,
	\begin{equation}\nonumber
		\begin{split}
			&\mathcal{L}\left(\left\{x\in\mathbb{I}\colon d_{1}=\sigma_{1},\cdots,\epsilon_{n-1}=\delta_{n-1},d_{n-1}=\sigma_{n-1},\epsilon_{n}=\epsilon_{n-1},y_{n}\leq c\right\}\right)\\
			=&\mathcal{L}\left(\left\{x\in\mathbb{I}\colon d_{1}=\sigma_{1},\cdots,\epsilon_{n-1}=\delta_{n-1},d_{n-1}=\sigma_{n-1},\epsilon_{n}=\epsilon_{n-1},T^{n-1}x\leq\frac{c}{d_{n-1}(x)-1}\right\}\right)\\
			=&\frac{c}{\sigma_{1}\cdots\sigma_{n-1}\left(\sigma_{n-1}-1\right)}.
		\end{split}
	\end{equation}
	Similarly,
	$$\mathcal{L}\left(\left\{x\in\mathbb{I}\colon d_{1}=\sigma_{1},\cdots,\epsilon_{n-1}=\delta_{n-1},d_{n-1}=\sigma_{n-1},\epsilon_{n}=-\epsilon_{n-1},y_{n}\leq c\right\}\right)=\frac{c}{\sigma_{1}\cdots\sigma_{n-1}\left(\sigma_{n-1}+1\right)}.$$
	It follows that
	\begin{equation}\label{yn intersect basic interval}
		\begin{split}
			&\mathcal{L}\left(\left\{x\in\mathbb{I}\colon d_{1}=\sigma_{1},\cdots,\epsilon_{n-1}=\delta_{n-1},d_{n-1}=\sigma_{n-1},y_{n}\leq c\right\}\right)\\
			=&\frac{c}{\sigma_{1}\cdots\sigma_{n-1}\left(\sigma_{n-1}-1\right)}+\frac{c}{\sigma_{1}\cdots\sigma_{n-1}\left(\sigma_{n-1}+1\right)}\\
			=&c\cdot\left|I_{\left(\sigma_{1},\cdots,\delta_{n-1},\sigma_{n-1}\right)}\right|.
		\end{split}
	\end{equation}
	Hence,
	\begin{equation}\nonumber
		\begin{split}
			\mathcal{L}\left(y_{n}\leq c\right)
			&=\bigcup_{\left(\sigma_{1},\cdots,\delta_{n-1},\sigma_{n-1}\right)\in\Sigma_{n-1}^{'}}\mathcal{L}\left(\left\{x\in\mathbb{I}\colon d_{1}=\sigma_{1},\cdots,\epsilon_{n-1}=\delta_{n-1},d_{n-1}=\sigma_{n-1},y_{n}\leq c\right\}\right)\\
			&=\bigcup_{\left(\sigma_{1},\cdots,\delta_{n-1},\sigma_{n-1}\right)\in\Sigma_{n-1}^{'}}c\cdot\left|I_{\left(\sigma_{1},\cdots,\delta_{n-1},\sigma_{n-1}\right)}\right|=c.
		\end{split}		
	\end{equation}
\end{proof}
The following lemma equivalently transforms a certain type of constraint on $y_{n}$ into a constraint on $d_{n}$. This facilitates the computation of the Lebesgue measure of the set formed by such constraints on $y_{n}$.
\begin{lemma}\label{Equivalent condition transformation}
	Let $r$ be a positive odd number. For any $x\in\mathbb{I}$, the following results hold.
	\begin{enumerate}[(i)]
		\item\label{Equivalent condition}
		$y_{1}(x)\leq\frac{1}{r}\quad\textnormal{if and only if}\quad d_{1}(x)\geq r$.
		\item\label{Equivalent condition transformation for 1} When $n\geq2$ and $\epsilon_{n}(x)=\epsilon_{n-1}(x)$, that is, $s_{n-1}(x)=1$, then $$y_{n}(x)\leq\frac{1}{r}\quad\textnormal{if and only if}\quad d_{n}(x)\geq r\left(d_{n-1}(x)-1\right)=r\left(d_{n-1}(x)-s_{n-1}(x)\right).$$
		\item\label{Equivalent condition transformation for -1} When $n\geq2$ and $\epsilon_{n}(x)=-\epsilon_{n-1}(x)$, that is, $s_{n-1}(x)=-1$, then $$y_{n}(x)\leq\frac{1}{r}\quad\textnormal{if and only if}\quad d_{n}(x)\geq r\left(d_{n-1}(x)+1\right)=r\left(d_{n-1}(x)-s_{n-1}(x)\right).$$
	\end{enumerate}
\end{lemma}
\begin{proof}
	Note that $\frac{1}{x}-1<d_{1}(x)<\frac{1}{x}+1$, $d_{1}$ is even and $r$ is odd. Then, the conclusion \eqref{Equivalent condition} holds.\\
	\indent Now, we prove \eqref{Equivalent condition transformation for 1}. It suffices to verify that $y_{n}(x)\leq\frac{1}{r}$ if and only if $d_{n}(x)\geq r\left(d_{n-1}(x)-1\right)$. Since $y_{n}(x)=\left(d_{n-1}(x)-1\right)T^{n-1}x$ and $T^{n-1}x>\frac{1}{d_{n}(x)+1}$, the necessity is immediate. For sufficiency, assume that $d_{n}(x)\geq r\left(d_{n-1}(x)-1\right)$. Since $d_{n-1}(x)$, $d_{n}(x)$ are even and $r$ is odd, it follows that $d_{n}(x)-1\geq r\left(d_{n-1}(x)-1\right)$. Hence, $$y_{n}(x)=\left(d_{n-1}(x)-1\right)T^{n-1}x<\frac{d_{n-1}(x)-1}{d_{n}(x)-1}\leq\frac{1}{r},$$ 
	where the first inequality follows from $T^{n-1}x<\frac{1}{d_{n}(x)-1}$.\\
	\indent The proof of \eqref{Equivalent condition transformation for -1} is analogous to that of \eqref{Equivalent condition transformation for 1}, and we omit the details.
\end{proof}
The following lemma shows that the sequence $y_{n}$ has a certain degree of independence.
\begin{lemma}\label{the propability of y1...yn}
	Let $\left\{r_{n}\right\}_{n\geq1}$ be a sequence of positive odd numbers. For any $n\in\mathbb{N}$, we have
	\begin{equation}\nonumber
		\mathcal{L}\left(y_{1}\leq\frac{1}{r_{1}},y_{2}\leq\frac{1}{r_{2}},\cdots,y_{n}\leq\frac{1}{r_{n}}\right)=\frac{1}{r_{1}r_{2}\cdots r_{n}}.
	\end{equation}
\end{lemma}
\begin{proof}
	The proof is carried out by induction. When $n=1$, $\mathcal{L}\left(y_{1}\leq\frac{1}{r_{1}}\right)=\mathcal{L}\left(x\leq\frac{1}{r_{1}}\right)=\frac{1}{r_{1}}$.\\
	\indent Now, assume that $\mathcal{L}\left(y_{1}\leq\frac{1}{r_{1}},y_{2}\leq\frac{1}{r_{2}},\cdots,y_{n}\leq\frac{1}{r_{n}}\right)=\frac{1}{r_{1}r_{2}\cdots r_{n}}$. By \cref{Equivalent condition transformation}, we have 
	\begin{equation}\nonumber
		\mathcal{L}\left(y_{1}\leq\frac{1}{r_{1}},\cdots,y_{n+1}\leq\frac{1}{r_{n+1}}\right)=\mathcal{L}\left(\left\{x\in\mathbb{I}\colon d_{1}\geq r_{1},\cdots,d_{n}\geq r_{n}\left(d_{n-1}-s_{n-1}\right),y_{n+1}\leq \frac{1}{r_{n+1}}\right\}\right).
	\end{equation}
	By \cref{yn intersect basic interval}, we have
	\begin{equation}\nonumber
		\begin{split}
			\mathcal{L}\left(y_{1}\leq\frac{1}{r_{1}},\cdots,y_{n+1}\leq\frac{1}{r_{n+1}}\right)
			&=\sum_{\left(\sigma_{1},\cdots,\delta_{n},\sigma_{n}\right)}\mathcal{L}\left(\left\{x\in\mathbb{I}\colon d_{1}=\sigma_{1},\cdots,\epsilon_{n}=\delta_{n},d_{n}=\sigma_{n},y_{n+1}\leq\frac{1}{r_{n+1}}\right\}\right)\\
			&=\sum_{\left(\sigma_{1},\cdots,\delta_{n},\sigma_{n}\right)}\frac{1}{r_{n+1}}\left|I_{\left(\sigma_{1},\cdots,\delta_{n},\sigma_{n}\right)}\right|\\
			&=\sum_{\left(\sigma_{1},\cdots,\delta_{n},\sigma_{n}\right)}\frac{1}{r_{n+1}}\mathcal{L}\left(\left\{x\in\mathbb{I}\colon d_{1}=\sigma_{1},\cdots,\epsilon_{n}=\delta_{n},d_{n}=\sigma_{n}\right\}\right)\\
			&=\frac{1}{r_{n+1}}\mathcal{L}\left(y_{1}\leq\frac{1}{r_{1}},\cdots,y_{n}\leq\frac{1}{r_{n}}\right)=\frac{1}{r_{1}r_{2}\cdots r_{n}r_{n+1}},
		\end{split}
	\end{equation}
	where the sums are taken over all $\left(\sigma_{1},\cdots,\delta_{n},\sigma_{n}\right)\in\Sigma_{n}^{'}$ satisfying $\sigma_{k}\geq r_{k}\left(\sigma_{k-1}-\frac{\delta_{n}}{\delta_{n-1}}\right)$ for all $2\leq k\leq n$ and $\sigma_{1}\geq r_{1}$.
	It follows by induction that the conclusion holds.
\end{proof}
The following result shows that $\left\{R_{n}\right\}_{n\geq1}$ is a sequence of approximately independent and identically distributed random variables.
\begin{lemma}\label{definition property Yn}
	For any $n\in\mathbb{N}$, let $Y_{n}$ be a function defined on $\mathbb{I}$, satisfying for any $x\in\mathbb{I}$:
	\begin{enumerate}[(i)]
		\item\label{Yn odd} $Y_{n}(x)$ is odd,
		\item if $n=1$, then $Y_{1}(x)\leq d_{1}(x)<Y_{1}(x)+2$,
		\item if $n\geq2$ and $\epsilon_{n}(x)=\epsilon_{n-1}(x)$, then $Y_{n}(x)\leq\frac{d_{n}(x)}{d_{n-1}(x)-1}<Y_{n}(x)+2$, and
		\item if $n\geq2$ and $\epsilon_{n}(x)=-\epsilon_{n-1}(x)$, then $Y_{n}(x)\leq\frac{d_{n}(x)}{d_{n-1}(x)+1}<Y_{n}(x)+2$.
	\end{enumerate}
	Then,
	\begin{enumerate}[(1)]
		\item\label{propability of Yn} $\mathcal{L}\left(Y_{n}\geq2k-1\right)=\frac{1}{2k-1}$ for any $n\in\mathbb{N}$ and $k\in\mathbb{N}$,
		\item\label{control of Yn} $\frac{1}{t+2}<\mathcal{L}\left(Y_{n}\geq t\right)\leq\frac{1}{t}$ for any $n\in\mathbb{N}$ and $t>0$, and
		\item\label{Yn independent} $\left\{Y_{n}\right\}_{n\geq1}$ is a sequence of independent and identically distributed random variables.
	\end{enumerate}
\end{lemma}
\begin{proof}
	Let $l\in\mathbb{N}$. When $l$ is odd, combining the definition of $Y_{n}$ with \cref{Equivalent condition transformation}, we can obtain $Y_{n}(x)\geq l$ if and only if $y_{n}(x)\leq\frac{1}{l}$ for any $x\in\mathbb{I}$. Similarly, when $l$ is even, $Y_{n}(x)\geq l$ if and only if $y_{n}(x)\leq \frac{1}{l+1}$ for any $x\in\mathbb{I}$. Then, by \cref{the probability of yn}, we have
	\begin{equation}\nonumber
		\mathcal{L}\left(Y_{n}\geq2k-1\right)=\mathcal{L}\left(y_{n}\leq\frac{1}{2k-1}\right)=\frac{1}{2k-1}.
	\end{equation} 
	Hence, $\left\{Y_{n}\right\}_{n\geq1}$ is identically distributed, and conclusion \eqref{propability of Yn} holds. Since $Y_{n}(x)$ is odd, it follows that \eqref{propability of Yn} implies \eqref{control of Yn}.
	We now prove that $\left\{Y_{n}\right\}_{n\geq1}$ are independent. It suffices to check that when $l_{1},\cdots,l_{n}$ are all odd, we have
	$$\mathcal{L}\left(Y_{1}\geq l_{1},Y_{2}\geq l_{2},\cdots,Y_{n}\geq l_{n}\right)=\mathcal{L}\left(Y_{1}\geq l_{1}\right)\mathcal{L}\left(Y_{2}\geq l_{2}\right)\cdots\mathcal{L}\left(Y_{n}\geq l_{n}\right).$$
	By \cref{the propability of y1...yn}, we have 
	$$\mathcal{L}\left(Y_{1}\geq l_{1},Y_{2}\geq l_{2},\cdots,Y_{n}\geq l_{n}\right)=\mathcal{L}\left(y_{1}\leq\frac{1}{l_{1}},y_{2}\leq\frac{1}{l_{2}}\cdots,y_{n}\leq\frac{1}{l_{n}}\right)=\frac{1}{l_{1}l_{2}\cdots l_{n}},$$
	which implies independence.
\end{proof}
\begin{proof}[Proof of \cref{Borel-Bernstain thm of ratio form}]
	First, assume that $\sum_{n=1}^{\infty}\frac{1}{\phi(n)}<\infty$. For any $x\in\mathbb{I}$ and $n\geq2$, we have
	\begin{equation}\nonumber
		Y_{n}(x)+2>\frac{d_{n}(x)}{d_{n-1}(x)+1}>\frac{d_{n}(x)}{2d_{n-1}(x)}=\frac{R_{n}(x)}{2},
	\end{equation}
	which implies that
	\begin{equation}\nonumber
		\left\{x\in\mathbb{I}\colon R_{n}(x)\geq\phi(n)\right\}\subset\left\{x\in\mathbb{I}\colon Y_{n}(x)\geq\frac{\phi(n)}{6}\right\}.
	\end{equation}
	It follows from \cref{definition property Yn} \eqref{control of Yn} that
	\begin{equation}\nonumber
		\sum_{n=1}^{\infty}\mathcal{L}\left(\left\{x\in\mathbb{I}\colon R_{n}(x)\geq\phi(n)\right\}\right)\leq\sum_{n=1}^{\infty}\mathcal{L}\left(\left\{x\in\mathbb{I}\colon Y_{n}(x)\geq\frac{\phi(n)}{6}\right\}\right)\leq\sum_{n=1}^{\infty}\frac{6}{\phi(n)}<\infty,
	\end{equation}
	By \cref{Borel--Cantelli lemma}, we get $\mathcal{L}\left(R(\phi)\right)=0$.\\
	\indent Now, suppose $\sum_{n=1}^{\infty}\frac{1}{\phi(n)}=\infty$. By \cref{definition property Yn} \eqref{control of Yn}, we get
	\begin{equation}\nonumber
		\sum_{n=1}^{\infty}\mathcal{L}\left(\left\{x\in\mathbb{I}\colon Y_{n}(x)\geq2\phi(n)\right\}\right)\geq\sum_{n=1}^{\infty}\frac{1}{2\phi(n)+2}=\infty.
	\end{equation}
	By \cref{definition property Yn} \eqref{Yn independent} and \cref{Borel--Cantelli lemma}, we get $\mathcal{L}\left(\left\{x\in\mathbb{I}\colon Y_{n}(x)\geq2\phi(n),\ \textnormal{  i.m. }n\right\}\right)=1$.
	For any $x\in\mathbb{I}$ and $n\geq2$, we have
	\begin{equation}\nonumber
		2R_{n}(x)=\frac{2d_{n}(x)}{d_{n-1}(x)}\geq\frac{d_{n}(x)}{d_{n-1}(x)-1}\geq Y_{n}(x),
	\end{equation}
	which implies that 
	\begin{equation}\nonumber
		\left\{x\in\mathbb{I}\colon Y_{n}(x)\geq2\phi(n),\ \textnormal{  i.m. }n\right\}\subset\left\{x\in\mathbb{I}\colon R_{n}(x)\geq\phi(n),\ \textnormal{  i.m. }n\right\}.
	\end{equation}
	Hence, $\mathcal{L}\left(R(\phi)\right)=1$.
\end{proof}
\begin{proof}[Proof of \cref{lebesguemeasureoflimsupRn/phin}]
	Let $K$ be an arbitrary positive number. If $\sum_{n=1}^{\infty}\frac{1}{\phi(n)}<\infty$, then, by replacing $\phi(n)$
	with $\frac{\phi(n)}{K}$ in \cref{Borel-Bernstain thm of ratio form}, we have $\mathcal{L}\left\{x\in\mathbb{I}\colon R_{n}(x)\geq \frac{\phi(n)}{K},\textnormal{ i.m. }n\right\}=0$. It follows $\mathcal{L}\left\{x\in\mathbb{I}\colon\varlimsup_{n\to\infty}\frac{R_{n}(x)}{\phi(n)}>\frac{1}{K}\right\}=0$. By the arbitrariness of $K$, we get $\mathcal{L}\left\{x\in\mathbb{I}\colon\varlimsup_{n\to\infty}\frac{R_{n}(x)}{\phi(n)}>0\right\}=0$, which implies
	\begin{equation}\nonumber
		\mathcal{L}\left\{x\in\mathbb{I}\colon\varlimsup_{n\to\infty}\frac{R_{n}(x)}{\phi(n)}=0\right\}=1.
	\end{equation}
	If $\sum_{n=1}^{\infty}\frac{1}{\phi(n)}=\infty$, then we have $\mathcal{L}\left\{x\in\mathbb{I}\colon R_{n}(x)\geq K\phi(n),\textnormal{ i.m. }n\right\}=1$ by replacing $\phi(n)$
	with $K\phi(n)$ in \cref{Borel-Bernstain thm of ratio form}. It follows that $\mathcal{L}\left\{x\in\mathbb{I}\colon\varlimsup_{n\to\infty}\frac{R_{n}(x)}{\phi(n)}\geq K\right\}=1$.
	By the arbitrariness of $K$, we obtain 
	\begin{equation}\nonumber
		\mathcal{L}\left\{x\in\mathbb{I}\colon\varlimsup_{n\to\infty}\frac{R_{n}(x)}{\phi(n)}=\infty\right\}=1.
	\end{equation}
\end{proof}
\begin{proof}[Proof of \cref{Borel--Cantelli lemma for Mphi}]
	Note that $R\left(\phi\right)\subset M\left(\phi\right)$. If $\sum_{n=1}^{\infty}\frac{1}{\phi(n)}=\infty$, then, by \cref{Borel-Bernstain thm of ratio form}, we have $\mathcal{L}\left(M\left(\phi\right)\right)\geq\mathcal{L}\left(R(\phi)\right)=1$.\\
	\indent Now, assume that $\sum_{n=1}^{\infty}\frac{1}{\phi(n)}<\infty$. Using the conditions that $\phi$ is non-decreasing and $\sum_{n=1}^{\infty}\frac{1}{\phi(n)}<\infty$, we can verify that $M\left(\phi\right)\subset R\left(\phi\right)$. Then, by \cref{Borel-Bernstain thm of ratio form}, $\mathcal{L}\left(M\left(\phi\right)\right)=0$.
\end{proof}
\begin{proof}[Proof of \cref{lebesguemeasureoflimsupMn/phin}]
	Using \cref{Borel--Cantelli lemma for Mphi}, one can establish the desired results in complete analogy with the proof of \cref{lebesguemeasureoflimsupRn/phin}. We omit the details.
\end{proof}
\subsection{Proofs of \cref{limsuplogRn-logn/loglogn,liminflogRn-logn/loglogn}}
\begin{proof}[Proof of \cref{limsuplogRn-logn/loglogn}]
	For any $0<\varepsilon<1$, by $\sum_{n=1}^{\infty}\frac{1}{n(\log n)^{1-\varepsilon}}=\infty$ and $\sum_{n=1}^{\infty}\frac{1}{n(\log n)^{1+\varepsilon}}<\infty$, we get, from \cref{Borel-Bernstain thm of ratio form}, 
	\begin{equation}\nonumber
		\mathcal{L}\left\{x\in\mathbb{I}\colon R_{n}(x)\geq n(\log n)^{1-\varepsilon},\textnormal{ i.m. }n\right\}=1,
	\end{equation}
	and
	\begin{equation}\nonumber
		\mathcal{L}\left\{x\in\mathbb{I}\colon R_{n}(x)\geq n(\log n)^{1+\varepsilon},\textnormal{ i.m. }n\right\}=0.
	\end{equation}
	It follows that for Lebesgue almost all $x\in\mathbb{I}$, 
	\begin{equation}\nonumber
		1-\varepsilon\leq\varlimsup_{n\to\infty}\frac{\log R_{n}(x)-\log n}{\log\log n}\leq1+\varepsilon.
	\end{equation}
	By the arbitrariness of $\varepsilon$, for Lebesgue almost all $x\in\mathbb{I}$, we have
	\begin{equation}\nonumber
		\varlimsup_{n\to\infty}\frac{\log R_{n}(x)-\log n}{\log\log n}=1.
	\end{equation}
	Note that both $n(\log n)^{1-\varepsilon}$ and $n(\log n)^{1+\varepsilon}$ are non-decreasing in $n$. Using \cref{Borel--Cantelli lemma for Mphi}, we can obtain $\varlimsup_{n\to\infty}\frac{\log M_{n}(x)-\log n}{\log\log n}=1$ for Lebesgue almost all $x\in(0,1)$.
\end{proof}
\begin{proof}[Proog of \cref{liminflogRn-logn/loglogn}]
	By \cref{definition property Yn}, we get, for any $n\in\mathbb{N}$, 
	$$\mathcal{L}(Y_{n}=1)=\mathcal{L}(Y_{n}\geq1)-\mathcal{L}(Y_{n}\geq3)=\frac{2}{3},$$
	which implies that $\sum_{n=1}^{\infty}\mathcal{L}(Y_{n}=1)=\infty$. Since 
	$\left\{Y_{n}\right\}_{n\geq1}$ is a sequence of independent random variables, it follows from \cref{Borel--Cantelli lemma} that
	\begin{equation}\nonumber
		\mathcal{L}\left\{x\in\mathbb{I}\colon Y_{n}(x)=1,\textnormal{ i.m. }n\right\}=1.
	\end{equation}
	According to \cref{definition property Yn}, note that $R_{n}(x)\leq5$ when $Y_{n}(x)=1$ for any $n\in\mathbb{N}$. Then, for Lebesgue almost all $x\in(0,1)$, we have 
	$$\varliminf_{n\to\infty}\frac{\log R_{n}(x)-\log n}{\log\log n}=-\infty.$$
	
	Now we prove the remaining conclusion. First, for any $x\in\mathbb{I}$ and $n\in\mathbb{N}$, we define
	\begin{equation}\nonumber
		U_{n}(x)\coloneqq\max\left\{Y_{i}(x)\colon1\leq i\leq n\right\}.
	\end{equation}
	By the definition of $Y_{n}$, one can verify that $\frac{1}{2}Y_{n}(x)\leq R_{n}(x)<2(Y_{n}(x)+2)\leq6Y_{n}(x)$ for any $x\in\mathbb{I}$ and $n\in\mathbb{N}$, which implies that 
	$\frac{1}{2}U_{n}(x)\leq M_{n}(x)<6U_{n}(x)$. Hence, it suffices to verify that for Lebesgue almost all $x\in\mathbb{I}$, we have
	\begin{equation}\nonumber
		\varliminf_{n\to\infty}\frac{\log U_{n}(x)-\log n}{\log\log n}=0.
	\end{equation}
	We split into two parts.\\
	\indent {\sc Upper bound}. Fix $\varepsilon>0$. For any $n\in\mathbb{N}$, define
	\begin{equation}\nonumber
		B_{n}\coloneqq\left\{U_{n}>n(\log n)^{\varepsilon}\right\}.
	\end{equation}
	By \cref{definition property Yn}, we have
	\begin{equation}\nonumber
		\mathcal{L}(B_{n})=\mathcal{L}\left(\bigcup_{i=1}^{n}\left\{Y_{i}>n(\log n)^{\varepsilon}\right\}\right)\leq n\cdot\mathcal{L}\left(\left\{Y_{1}>n(\log n)^{\varepsilon}\right\}\right)\leq n\cdot\frac{1}{n(\log n)^{\varepsilon}}=\frac{1}{(\log n)^{\varepsilon}}.
	\end{equation}
	Take $\alpha>\frac{1}{\varepsilon}$. For any $k\in\mathbb{N}$, let $n_{k}\coloneqq\left\lceil e^{k^{\alpha}}\right\rceil$. Then, we have
	\begin{equation}\nonumber
		\mathcal{L}(B_{n_{k}})\leq\frac{1}{(\log n_{k})^{\varepsilon}}\leq\frac{1}{k^{\alpha\varepsilon}},
	\end{equation}
	which implies that $\sum_{k=1}^{\infty}\mathcal{L}(B_{n_{k}})<\infty$. By \cref{Borel--Cantelli lemma}, we get
	\begin{equation}\nonumber
		\mathcal{L}\left(\left\{B_{n_{k}}\ \ \textnormal{i.m.}\ k\right\}\right)=\mathcal{L}\left(\left\{U_{n_{k}}>n_{k}(\log n_{k})^{\varepsilon},\ \ \textnormal{i.m.}\ k\right\}\right)=0.
	\end{equation}
	That is, there exists $\Omega_{\varepsilon}\subset\mathbb{I}$, such that $\mathcal{L}(\Omega_{\varepsilon})=1$ and for each $x\in\Omega_{\varepsilon}$, we have $U_{n_{k}}(x)\leq n_{k}(\log n_{k})^{\varepsilon}$ for all sufficiently large $k$. Hence, for any $x\in\Omega_{\varepsilon}$, we get 
	\begin{equation}\nonumber
		\varliminf_{n\to\infty}\frac{\log U_{n}(x)-\log n}{\log\log n}\leq\varliminf_{k\to\infty}\frac{\log U_{n_{k}}(x)-\log n_{k}}{\log\log n_{k}}\leq\varepsilon.
	\end{equation}
	For any $j\in\mathbb{N}$, replace $\varepsilon$ by $\frac{1}{j}$ and let $\Omega_{0}\coloneqq\bigcap_{j=1}^{\infty}\Omega_{\frac{1}{j}}$. One can check that $\mathcal{L}(\Omega_{0})=1$ and for any $y\in\Omega_{0}$, 
	\begin{equation}\nonumber
		\varliminf_{n\to\infty}\frac{\log U_{n}(y)-\log n}{\log\log n}\leq0.
	\end{equation}
	\indent {\sc Lower bound}. Fix $\varepsilon>0$. For any $n\in\mathbb{N}$, define
	\begin{equation}\nonumber
		C_{n}=\left\{U_{n}<\frac{n}{(\log n)^{\varepsilon}}\right\}.
	\end{equation}
	By \cref{definition property Yn}, we have
	\begin{equation}\nonumber
		\mathcal{L}(C_{n})=\mathcal{L}\left(\bigcap_{i=1}^{n}\left\{Y_{i}<\frac{n}{(\log n)^{\varepsilon}}\right\}\right)=\left(\mathcal{L}\left\{Y_{1}<\frac{n}{(\log n)^{\varepsilon}}\right\}\right)^{n}\leq\left(1-\frac{1}{\frac{n}{(\log n)^{\varepsilon}}+2}\right)^{n}.
	\end{equation}
	For any $k\in\mathbb{N}$, let $m_{k}\coloneqq\lfloor e^{k}\rfloor$. By $\lim_{k\to\infty}\frac{\mathcal{L}(C_{m_{k}})}{e^{-k^{\varepsilon}}}=1$ and $\sum_{k=1}^{\infty}e^{-k^{\varepsilon}}<\infty$, then, $\sum_{k=1}^{\infty}\mathcal{L}(C_{m_{k}})<\infty$. It follows from \cref{Borel--Cantelli lemma} that
	\begin{equation}\nonumber
		\mathcal{L}\left(\left\{C_{m_{k}}\ \ \textnormal{i.m.}\ k\right\}\right)=\mathcal{L}\left(\left\{U_{m_{k}}<\frac{m_{k}}{(\log m_{k})^{\varepsilon}},\ \ \textnormal{i.m.}\ k\right\}\right)=0.
	\end{equation}
	That is, there exists $\Delta_{\varepsilon}\subset\mathbb{I}$, such that $\mathcal{L}(\Delta_{\varepsilon})=1$ and for each $x\in\Delta_{\varepsilon}$, we have $U_{m_{k}}(x)\geq \frac{m_{k}}{(\log m_{k})^{\varepsilon}}$ for all sufficiently large $k$. When $n$ is sufficiently large, the integer $k$ satisfying $m_{k}\leq n<m_{k+1}$ is also sufficiently large. Then, for any $x\in\Delta_{\varepsilon}$ and sufficiently large $n$, we have 
	\begin{equation}\nonumber
		U_{n}(x)\geq U_{m_{k}}(x)\geq\frac{m_{k}}{(\log m_{k})^{\varepsilon}}\geq\frac{m_{k+1}}{e^{2}\cdot(\log m_{k})^{\varepsilon}}>\frac{n}{e^{2}\cdot(\log n)^{\varepsilon}},
	\end{equation}
	which implies that 
	\begin{equation}\nonumber
		\varliminf_{n\to\infty}\frac{\log U_{n}(x)-\log n}{\log\log n}\geq-\varepsilon.
	\end{equation}
	For any $j\in\mathbb{N}$, substitute $\frac{1}{j}$ for $\varepsilon$, and let $\Delta_{0}\coloneqq\bigcap_{j=1}^{\infty}\Delta_{\frac{1}{j}}$. One can check that $\mathcal{L}(\Delta_{0})=1$ and for any $y\in\Delta_{0}$, 
	\begin{equation}\nonumber
		\varliminf_{n\to\infty}\frac{\log U_{n}(y)-\log n}{\log\log n}\geq0.
	\end{equation}
	The proof is completed.
\end{proof}
\begin{proof}[Proof of \cref{limlogMn/logn}]
	Combining \cref{limsuplogRn-logn/loglogn} and \cref{liminflogRn-logn/loglogn}, the conclusion is immediate.
\end{proof}

\noindent\textbf{Acknowledgements } I am grateful to my advisor, Professor Lingmin Liao, for many helpful suggestions. I am also grateful to Professor Lulu Fang for providing reference \cite{MR1180497}.
   
\noindent\textbf{Data Availability } Data sharing notapplicable to this article as no datasets were generated or analysed during the current study.

\noindent\textbf{Funding } No funding was received for this study.

\section*{Declarations}
\noindent\textbf{Conflicts of Interest } The author declares no conflict of interest.

\bibliographystyle{plain}
\bibliography{bmyref.bib}
\end{document}